\documentclass[leqno,final]{siamart171218}

\usepackage{amsmath,dsfont}
\usepackage{graphicx,color}
\usepackage{amssymb}

\usepackage{algorithm}
\usepackage{algorithmic}
\usepackage{float}
\usepackage{enumitem}   

\def\E{\mathbb{E}}

\newtheorem{remark}{Remark}
\newcommand{\mE}{\mathbb{E}}

\title{Optimal Order Time Discretizations for Stochastic Semilinear Wave Equations with Multiplicative Noise}

\author{
	Xiaobing Feng\thanks{Department of Mathematics, The University of Tennessee, Knoxville, TN 37996, U.S.A. This author was partially supported by the NSF grant DMS-2309626. ({\tt xfeng@utk.edu}).}  
	\and
	Yukun Li\thanks{Department of Mathematics, University of Central Florida, Orlando, FL, 32816, U.S.A. This author was partially supported by the NSF grant DMS-2110728. ({\tt yukun.li@ucf.edu}).}
	\and
	Liet Vo\thanks{School of Mathematical and Statistical Sciences, The University of Texas Rio Grande Valley, Edinburg, TX 78541, U.S.A.  ({\tt liet.vo@utrgv.edu}).} 
}

\begin{document}
	\maketitle
	
	\begin{abstract}
		This paper is concerned with developing and analyzing two novel implicit temporal discretization methods for the stochastic semilinear wave equations with multiplicative noise.  The proposed methods are natural extensions of well-known time-discrete schemes for deterministic wave equations, hence, they are easy to implement. It is proved that both methods are energy-stable. Moreover, the first method is shown to converge with the linear order in the energy norm, while the second method converges with the $\mathcal{O}(\tau^{\frac32})$ order in the $L^2$-norm, which is optimal with respect to the time regularity of the solution to the underlying stochastic PDE. The convergence analyses of both methods, which are different and quite involved, require some novel numerical techniques to overcome difficulties caused by the nonlinear noise term and the interplay between nonlinear drift and diffusion.  Numerical experiments are provided to validate the sharpness of the theoretical error estimate results. 
\end{abstract}

\begin{keywords}
Stochastic wave equations, It\^o integral, variational solutions, H\"older continuity,  time discretization, Monte Carlo method, error estimates. 
\end{keywords}

\begin{AMS}
65N12, 
65N15, 
65N30 
\end{AMS}

\section{Introduction}\label{sec-1}
In this paper, we construct and analyze two optimal order time-discrete methods for the following initial-boundary value problem for the stochastic semilinear wave equation: 
\begin{subequations}\label{eq20230423_1}
\begin{alignat}{2}
	du_t - \Delta u &=F\left(u\right)+\sigma\left(u\right) dW(t) &&\quad \text { in }(0, T) \times \mathcal{D}, \label{spde} \\ 
	u(0, \cdot) &=u_0(\cdot), \quad u_t(0, \cdot) =v_0(\cdot) &&\quad \text { in } \mathcal{D}, \\ 
	u(t, \cdot) &=0 &&\quad\text { on } \partial \mathcal{D}\times\in(0, T),
\end{alignat}
\end{subequations}
where $\mathcal{D} \subset \mathbb{R}^d\ (1 \leq d \leq 3)$ is a bounded Lipschitz domain, $\{W(t)\}_{t \geq 0}$ denotes an $n$-dimensional Wiener process defined on a filtered probability space $(\Omega, \mathcal{F}, \mathbb{F}, \mathbb{P})$ with $\mathbb{F}=\left\{\mathcal{F}_t\right\}_{0 \leq t \leq T}$, the initial data $u_0$ and $v_0$ are two given $\mathcal{F}_0$-measurable vector fields, $F(u)$ and $\sigma(u)$  two given functionals (see Section \ref{sec-2} for precise assumptions on $F$ and $\sigma$).  We note that 
equation \eqref{spde} is often rewritten in the following mixed form:
\begin{subequations} \label{sac_s_mixed}
\begin{align}
	du &=v dt, \\
	dv - \Delta u &=F\left(u\right)+\sigma\left(u\right) dW(t) , 
\end{align}
\end{subequations}
where the auxiliary variable $v:=u_t$ defines the velocity when $u$ represents the displacement.  We shall frequently use the above mixed form throughout this paper. 

The wave propagation problems have wide applications in various fields and have been intensively studied for decades.   There are two defining factors that determine wave propagation, one is the source that generates the wave and the other is the medium in which the wave propagates. The above scalar wave equation indicates that the wave source is acoustic (hence, the equation is also called a stochastic acoustic wave equation).  The inclusion of the noise term is a mathematical formalism of uncertainties in the medium and/or in the source. 
We refer the readers to \cite{adjerid2011discontinuous,baccouch2012local,baker1976error,chou2014optimal,chung2006optimal,chung2009optimal,dupont19732,falk1999explicit,grote2006discontinuous,grote2009optimal,monk2005discontinuous,riviere2003discontinuous,safjan1993high,sun2021,xing2013energy,zhong2011numerical} and the references therein for extensive studies of deterministic wave propagation problems and to  \cite{chow2015stochastic,chow2002stochastic,chow2006asymptotics,chow2009nonlinear,dalang2009stochastic,dalang1998stochastic,millet2000stochastic,millet1999stochastic} for mathematical analyses for the semilinear stochastic wave equation.

Because there are no closed-form solutions in general, numerical approximations of the above stochastic wave equation, which is the only means of solving the stochastic PDEs, have garnered a lot of attention in recent years and various numerical methods were proposed and analyzed.  The case of additive 
noise (i.e., $\sigma(u)$ is independent of $u$) was well addressed in  \cite{cohen2013trigonometric,cui2019strong,gubinelli2018renormalization,hausenblas2010weak,kovacs2010finite} and the case of multiplicative noise was intensively investigated in  \cite{Anton2016full,cohen2018numerical,cohen2015fully,feng2022higher,hong2021energy,li2022finite,quer2006space}. 
Among those results, we note that  the strong convergence of order $\mathcal{O}(\tau^{\frac12})$ 
($\tau>0$ stands for the time mesh size) in the $L^2$-norm was obtained in \cite{quer2006space,walsh2006numerical}  for time approximations of  the displacement $u$ in the case of functional-type multiplicative noise (i.e., $\sigma$ only depends on $u$) and in \cite{li2022finite} where polynomial drift $F$,  which is not globally Lipschitz continuous,  was considered.  
Similar results were also obtained in \cite{feng2022higher} for more general drift and diffusion terms which may depend on both $u$ and $v$.  
The $\mathcal{O}(k)$ order time-discrete schemes were first obtained in \cite{cohen2013trigonometric} for the linear stochastic wave equation (i.e., $F(u)$ is linear) and were recently extended to the general case of nonlinear diffusion $\sigma$ and drift $F$  in  \cite{feng2022higher}. 

A natural question that was raised in  \cite{feng2022higher} asks what the optimal rate of convergence should be for temporal approximations of problem \eqref{eq20230423_1}. Here the optimality is with respect to the regularity of $u$ in time $t$.  Since $v=u_t$ is expected to be $\frac12$-H\"older continuous with respect to $t$, as a result, $u$ belongs to $H^{\frac32} (0,T; L^2(D))$ a.s. (see Section \ref{sec-2} for the detailed space definitions). Therefore, the optimal rate of convergence in the $L^2$-norm for temporal approximations is expected to be   
$\mathcal{O}(\tau^{\frac32})$.  One of the primary goals of  \cite{feng2022higher} was to construct 
such a time-discrete scheme and it was successfully achieved.  A special time-discrete scheme was developed in \cite{feng2022higher}   by using an operator splitting technique that splits the Laplacian into two parts and approximates each part at different time levels. It was proved that the proposed scheme converges 
in the energy norm with $\mathcal{O}(\tau^{\frac12})$ rate in the general case and in the $L^2$-norm with optimal  $\mathcal{O}(\tau^{\frac32})$ rate when the drift and diffusion only depend on $u$.   
In this paper, we continue addressing this interesting optimality question raised in \cite{feng2022higher}  and restated above. However, here our focus is on constructing simpler and more natural, especially optimal time-discrete schemes for problem \eqref{eq20230423_1}.

The primary goal of this paper is to develop and analyze two novel optimal order time-discrete schemes for problem \eqref{eq20230423_1}. Specifically, we first design a scheme using the stochastic Taylor's formula that is energy-stable and optimally convergent in the energy norm. This is done based on the mixed form \eqref{sac_s_mixed} of the stochastic wave equation.  Our second scheme also maintains the energy stability, in addition, it is proved to converge with the optimal rate $\mathcal{O}(\tau^{\frac32})$ in the $L^2$-norm. The idea for designing the second scheme is quite natural because it is constructed by using the structure of strong solutions of the stochastic wave equation. 

The remainder of this paper is organized as follows. In Section \ref{sec-2}, we introduce some notation, and assumptions on the drift and diffusion terms, and quote some known H\"{o}lder continuity results for the PDE solution. In Section \ref{sec-3}, we present two time-discrete schemes for problem  \eqref{eq20230423_1} and prove some energy stability results for both schemes.  Section \ref{sec-4} is devoted to convergence analysis. The highlights of the section are to establish the error estimates in both the energy norm and the $L^2$-norm. Finally, in Section \ref{sec-5}, we provide several numerical experiments to validate our theoretical error estimate results. 

\section{Preliminaries}\label{sec-2}
\subsection{Notations}

Standard function and space notation are adopted in this paper. In particular, ${ H}^{k}(D)$, for $k \geq 1$, denote the standard Sobolev spaces, and  ${ H}^1_0(D)$ denotes the subspace of ${H}^1(D)$ consisting of those functions which have zero trace on $\partial D$. $(\cdot,\cdot):=(\cdot,\cdot)_D$ and $\Vert \cdot \Vert_{L^2}$ denote respectively the standard $L^2$-inner product and its  induced norm. 	$C$ will be used to denote a generic constant that is independent of the mesh parameters $h$ and $\tau$.  

Moreover, let $(\Omega,\mathcal{F}, \{\mathcal{F}_t\},\mathbb{P})$ be a filtered probability space with the probability measure $\mathbb{P}$, the 
$\sigma$-algebra $\mathcal{F}$ and the continuous filtration $\{\mathcal{F}_t\} \subset \mathcal{F}$. For a random variable $v$ 
defined on $(\Omega,\mathcal{F}, \{\mathcal{F}_t\},\mathbb{P})$,
${\mathbb E}[v]$ denotes the expected value of $v$. 
For a vector space $X$ with norm $\|\cdot\|_{X}$,  and $1 \leq p < \infty$, we define the Bochner space
$\bigl(L^p(\Omega;X); \|v\|_{L^p(\Omega;X)} \bigr)$, where
$\|v\|_{L^p(\Omega;X)}:=\bigl({\mathbb E} [ \Vert v \Vert_X^p]\bigr)^{\frac1p}$.
Let $\mathcal{K}, \mathcal{H}$ be two separable Hilbert spaces, we use 
$\mathcal{L}_m(\mathcal{K}, \mathcal{H})$ to denote the space of all multi-linear maps from $\mathcal{K} \times \cdots \times \mathcal{K}$ ($m$-times) to $\mathcal{H}$ for $m \geq 1$.  For a mapping $\Psi: H_0^1 \times {H}_0^1 \rightarrow {L}^2$, we introduce the notation $D_u \Psi(u, v) \in \mathcal{L}\left({H}_0^1, {L}^2\right)$ for the Gateaux derivative with respect to  $u$, whose action is seen as
$$
\xi \mapsto D_u \Psi(u, v)(\xi) \qquad \text { for } \xi \in {H}_0^1(D) .
$$

We denote the second derivative with respect to $u$ by $D_u^2 \Psi(u, v) \in \mathcal{L}_2\left({H}_0^1, {L}^2\right)$, whose action can be seen as
$$
(\xi, \zeta) \mapsto D_u^2 \Psi(u, v)(\xi, \zeta):=\left(D_u^2 \Psi(u, v) \xi\right)(\zeta) \quad \text { for }(\xi, \zeta) \in\left[{H}_0^1(D)\right]^2 .
$$

Next, 	let $N>>1$ be a positive integer and $\tau:=T/N$. Let $\{t_n\}_{n=0}^N$ be a uniform partition of  the interval $[0,T]$ with mesh size $ \tau$. 	We define the useful notation $\overline{\Delta}W_{n}$ and $\widetilde{\Delta}W_n$ as follows:
\begin{align*}
\overline{\Delta}W_{n} &:= W(t_{n+1}) - W(t_n), \\ 
\widetilde{\Delta}W_n &:= \int_{t_n}^{t_{n+1}}[W(t_{n+1})-W(s)]ds\\
&= \tau W(t_{n+1}) - \int_{t_n}^{t_{n+1}} W(s)\, ds. 
\end{align*}
In addition, following  \cite{feng2022higher},  we approximate the last integral with higher accuracy as follows:
\begin{align}\label{approx_integral}
\int_{t_n}^{t_{n+1}} W(s)\, ds \approx \tau^3 \sum_{\ell=1}^{\tau^{-2}} W(t_{n,\ell}),
\end{align}
where $\{W({t_{n,\ell}})\}_{\ell=1}^{\tau^{-2}}$ is the piecewise affine approximation of $W$ on  $[t_n, t_{n+1}]$ over an equidistant mesh $\{t_{n,\ell}\}_{\ell=1}^{\tau^{-2}}$ with mesh size $\tau^3= t_{n,\ell + 1} - t_{n,\ell}$.

Thus, we obtain the following high order approximation of $\widetilde{\Delta}W_n$ (see Lemma \ref{increment}):
\begin{align}
\widehat{\Delta} W_n &:= \tau W(t_{n+1}) - \tau^3 \sum_{\ell = 1}^{\tau^{-2}} W(t_{n,\ell}).
\end{align}
We will frequently write $	\widetilde{\Delta} W_n \approx \widehat{\Delta} W_n$ in the subsequent sections. 



\subsection{Useful facts} 
In this subsection, we collect a few useful facts which will be crucially used in the proofs of our convergence results in the subsequent sections.  The first one is the following refined quadrature rule \cite[Theorem 2]{Dragomir}.

\begin{lemma}
Let $\phi \in C^{1,\alpha} ([0, T]; \mathbb{R})$ for some $\alpha \in (0, 1]$. Then, there exists a constant $C_e >0$ such that
\begin{align*}
	\biggl|\frac{\phi(0)  + \phi(T)}{2} - \frac{1}{T} \int_0^T \phi(\xi)\, d\xi\biggr| \leq \frac{C_e}{(\alpha + 2)(\alpha + 3)} \, T^{1+\alpha},
\end{align*}
where the constant $C_e$ satisfies
\begin{align*}
	|\phi'(t) - \phi'(s)| \leq C_e |t-s|^{\alpha}.
\end{align*}
\end{lemma}

Next, we state a few properties of $\overline{\Delta}W_{n}$, $\widetilde{\Delta}W_n$, and $ \widehat{\Delta} W_n$ which are defined in the previous subsection.  Their proofs can be found in \cite[Remark 1 and 2]{feng2022higher}.

\begin{lemma}\label{increment} 
	There hold the following estimates for the increments $\overline{\Delta}W_{n}$, $\widetilde{\Delta}W_n$, and $ \widehat{\Delta} W_n$:
	\begin{enumerate}
		\item[\rm (i)] $\mE\bigl[|\widetilde{\Delta}W_n - \widehat{\Delta}W_n|^2\bigr] \leq C\tau^5$.
		\item[\rm (ii)] $\mE\bigl[|\widetilde{\Delta}W_n|^2\bigr] + \mE\bigl[|\widehat{\Delta}W_n|^2\bigr] \leq C\tau^3$.
	\end{enumerate}
\end{lemma}

\subsection{Structure assumptions} 
Below we state the structure assumptions on nonlinear functions $F$ and $\sigma$ that were alluded in Section \ref{sec-1} and will be used in our convergence analysis. We note that no attempts are made to make these conditions sharp although we believe that they can be weakened but leave the task to the interested reader to exploit. 	

\medskip
\begin{enumerate}
\item[{\bf(A1)}]  Let $F: {H}_0^1(D) \rightarrow {L}^2(D)$ and $\sigma: {H}_0^1(D) \rightarrow {H}_0^1(D)$. Assume that there exists a constant $C>0$ such that
$$ \|F(u)\|_{\mathcal{K}} + \|\sigma(u)\|_{\mathcal{K}} \leq C(1 + \|u\|_{\mathcal{K}}),
$$
where $\mathcal{K} = L^2(D)$ or $H^1(D)$.
\item[{\bf(A2)}] Assume there is a constant $C>0$ such that 
$$
\left\|F(u)-F(\tilde{u})\right\|_{\mathcal{K}}+\left\|\sigma(u)-\sigma(\tilde{u})\right\|_{\mathcal{K}} \leq C\|u-\tilde{u}\|_{\mathcal{K}},
$$
where $\mathcal{K} = L^2(D)$ or $H^1(D)$.
\item[{\bf (A3)}] Assume there exists a constant $C>0$ such that 
$$
\left\|D_u^m F(\cdot)\right\|_{L^{\infty}\left({H}_0^1 ; \mathcal{L}_m\left({H}_0^1, {L}^2\right)\right)}+\left\|D_u^m \sigma(\cdot)\right\|_{L^{\infty}\left({H}_0^1 ; \mathcal{L}_m\left({H}_0^1, {H}_0^1\right)\right)} \leq C, \quad m=1,2,3 .
$$			
\end{enumerate}
It is easy to see that {\bf(A1)} essentially assumes that $F$ and $\sigma$ have linear growth for large input, {\bf(A2)} assumes they are Lipschitz continuous, and {\bf(A3)} imposes a boundedness condition on the derivatives of the two functions. 

\subsection{Variational solutions and their properties} 
%

We define variational solutions for problem \eqref{eq20230423_1} based on its mixed form \eqref{sac_s_mixed}.

\begin{definition}
$(u, v)$ is called a variational solution to problem  \eqref{sac_s_mixed} if
\begin{enumerate}
\item[\rm (i)]  $(u, v)$ is an ${H}_0^1(D) \times {L}^2(D)$-valued $\left\{\mathcal{F}_t\right\}$-adapted process;
\item[\rm (ii)] $(u, v) \in L^2\left(\Omega ; C\left([0, T] ; {H}_0^1(D)\right)\right) \times L^2\left(\Omega ; C\left([0, T] ; {L}^2(D)\right)\right)$ such that  there hold $\mathbb{P}$-a.s for each $t \in[0, T]$  
\begin{subequations}\label{sac_mixed_weakform}
\begin{align}
	(u(t), \phi) & =\int_0^t(v, \phi) \mathrm{d} s+\left(u_0, \phi\right) \quad \forall \phi \in {L}^2(D), \\
	(v(t), \psi) & =-\int_0^t[(\nabla u, \nabla \psi)+(F(u), \psi)] \mathrm{d} s\\\nonumber
	&\qquad+\int_0^t(\psi, \sigma(u) \mathrm{d} W(s))+\left(v_0, \psi\right) \quad \forall \psi \in {H}_0^1(D);
\end{align}
\end{subequations}

\item[\rm (iii)] there exists a constant $C=C(T, D, u_0)>0$ such that there holds $\mathbb{P}$-a.s.
$$
\mathbb{E}\left[\sup _{0 \leq t \leq T} \|v(t)\|^2_{L^2} + \|\nabla u(t)\|^2_{L^2}\right] \leq C .
$$
\end{enumerate}
\end{definition} 

For the existence and uniqueness of such a variational solution, we refer the reader to \cite[Section 6.8, Theorem 8.4]{chow2015stochastic}.

Next, we cite the following H\"older continuity estimates from \cite[Lemma 3.2]{feng2022higher}.
\begin{lemma}\label{lemma2.3}
Suppose $(\mathbf{A 1})$, ($\mathbf{A 2}$), and ($\mathbf{A 3}$) hold. Let $(u, v)$ be the variational solution of \eqref{sac_s_mixed}.   Then for any $s, t \in[0, T]$, $p \geq 1$, there hold  the  following statements:
\begin{enumerate}
\item[\rm (i)] 	if $(u_0,v_0) \in L^2(\Omega; H^1_0(D)\times L^2(D))$, then
$$
\left(\mathbb{E}\left[\sup _{s \leq r \leq t}\|u(r)-u(s)\|_{{L}^2}^{2 p}\right]\right)^{\frac{1}{2p}} \leq C|t-s|;
$$
\item[\rm (ii)]  if $(u_0,v_0) \in L^2(\Omega; (H^1_0(D)\cap H^2(D))\times H^1_0(D))$, then
$$
\left(\mathbb{E}\left[\sup _{s \leq r \leq t}\|u(r)-u(s)\|_{{H}^1}^{2 p}\right]\right)^{\frac{1}{2p}}+\mathbb{E}\left[\sup _{s \leq r \leq t}\|v(r)-v(s)\|_{{L}^2}^2\right] \leq C|t-s|;
$$
\item[\rm (iii)] if $(u_0,v_0) \in L^2(\Omega; (H^1_0(D)\cap H^3(D))\times (H^1_0(D))\cap H^2(D))$, then
$$
\left(\mathbb{E}\left[\sup _{s \leq r \leq t}\|u(r)-u(s)\|_{{H}^2}^{2 p}\right]\right)^{\frac{1}{2p}}+\mathbb{E}\left[\sup _{s \leq r \leq t}\|v(r)-v(s)\|_{{H}^1}^2\right] \leq C|t-s|;
$$
\item[\rm (iv)]  if $(u_0,v_0) \in L^2(\Omega; (H^1_0(D)\cap H^4(D))\times (H^1_0(D))\cap H^3(D))$, then
$$
\left(\mathbb{E}\left[\sup _{s \leq r \leq t}\|u(r)-u(s)\|_{{H}^3}^{2 p}\right]\right)^{\frac{1}{2p}}+\mathbb{E}\left[\sup _{s \leq r \leq t}\|v(r)-v(s)\|_{{H}^2}^2\right] \leq C|t-s|.
$$
\end{enumerate}

\end{lemma}
Since the variational solution has no (weak) derivatives in time, the above H\"older continuities play the crucial role of substituting time derivatives in the error analysis to be given in the subsequent sections.

\section{Time-discrete approximation schemes}\label{sec-3} 

\subsection{Formulation of schemes}
In this section, we propose two time-discrete schemes for problem  \eqref{eq20230423_1} which can be compactly written as the following $\theta$-scheme. 
\smallskip

{\bf $\theta$-scheme:}\, Let $N>>1$ be a positive integer and $\tau:=T/N$. Let $\{t_n\}_{n=0}^N$ be a uniform partition of  the interval $[0,T]$ with mesh size $ \tau$. For $\theta=0$ and $\theta = \frac12$, find $\mathcal{F}_{t_n}$ adapted process $\{(u^n,v^n)\}_{n=0}^N$ such that there hold $\mathbb{P}$-almost surely
\begin{subequations} \label{theta_scheme}
\begin{align}
&(u^{n+1}-u^n,\varphi) = \tau (v^{n+1},\varphi)-(\sigma(u^n)\widehat{\Delta}W_n,\varphi) \quad \forall \varphi \in H_0^1(D),\label{eq20230703_30}\\
&(v^{n+1} - v^n, \phi) +  \tau(\nabla u^{n,\theta}, \nabla \phi)  = \tau(F^{n,\theta}, \phi)\label{eq20230703_31}
+ (\sigma(u^n)\overline{\Delta} W_n, \phi)  \\\nonumber
&\hskip 2.0in + 2\theta (D_u\sigma(u^n)v^n\widehat{\Delta}W_n, \phi) \quad \forall \, \phi \in H^1_0(D),
\end{align}
\end{subequations}
where 
\begin{align*}
u^{n,\theta} := (1 - \theta) u^{n+1} + \theta u^{n-1},\qquad  F^{n,\theta} := (1 - \theta) F(u^{n+1}) + \theta  F(u^{n-1}).
\end{align*}

\begin{remark}\label{remark1} \
\begin{itemize}
\item[(a)] Equation \eqref{eq20230703_31} is essentially adapted from the time-discrete schemes of \cite{dupont19732} for the deterministic wave equation and a high order approximation of It\^o's stochastic integral (also see \cite{feng2022higher}).  On the other hand, at a glance, \eqref{eq20230703_30} may look unusual because it intends to approximate equation $u_t=v$ which does not directly involve any stochastic derivative. However, a closer look shows that it is quite natural in order to design an optimal order scheme. First, if one approximates $u_t=v$ by $u^{n+1}-u^n =\tau v^n$ as done in \cite{feng2022higher}, it is easy to see that such an approximation is only of first order in $\tau$. Second, to obtain a higher order approximation formula, we integrate both equations in \eqref{sac_s_mixed}  to get 
\begin{align*}
u(t_{n+1})-u(t_n) &= \int_{t_n}^{t_{n+1}} v(s) \, ds,   \\
v(s) - v(t_n) &=  \int_{t_n}^s \bigl[  \Delta u(\lambda) + F(u(\lambda)) \bigr]\, d\lambda   +\int_{t_n}^s \sigma(u(\lambda)) dW(\lambda)
\end{align*}
for $s\in [t_n, t_{n+1}]$.  Substituting the second quation into the first one we get 
\begin{align}\label{enhanced_holder}
u(t_{n+1})-u(t_n) = \tau v(t_n) &+ \int_{t_n}^{t_{n+1}} \int_{t_n}^s \bigl[  \Delta u(\lambda) + F(u(\lambda))   \bigr]\, d\lambda ds \\
& +\int_{t_n}^{t_{n+1}}  \int_{t_n}^s     \sigma(u(\lambda)) dW(\lambda) ds. \nonumber
\end{align}
Now, since 
\[
\int_{t_n}^{t_{n+1}}   \int_{t_n}^s \bigl[  \Delta u(\lambda) + F(u(\lambda))     \bigr]d\lambda ds =\mathcal{O}(\tau^2), 
\]
and 
\[
\int_{t_n}^{t_{n+1}}   \int_{t_n}^s     \sigma(u(\lambda)) dW(\lambda) ds = \sigma(u(t_n)) \int_{t_n}^{t_{n+1}}     \bigr[W(s)  -W(t_n) \bigr] \, ds
+ \mathcal{O}(\tau^2), 
\]
then we get the following truncation error equation:
\begin{align*} 
u(t_{n+1})-u(t_n) =  \tau v(t_n)  + \sigma(u(t_n)) \widehat{\Delta}W_n + \mathcal{O}(\tau^{\frac32}),
\end{align*}
which thus heuristically explains the approximation \eqref{eq20230703_30}. 

\item[(b)] 	When $\theta = \frac12$, equation \eqref{eq20230703_31} iterates from $n = 1$ instead of $n=0$. Therefore, $u^1$ is also an initial value and must be specified.  Following \cite{dupont19732,	feng2022higher}  we choose $u^1$ as follows:
\begin{align}\label{eq3.3}
u^1 = \tau v^0 + u^0 -\frac{\tau^2}{2}\Delta u^0 - \frac{\tau^2}{2}F(u^0) - \sigma(u^0)\widehat{\Delta} W_0 + \tau\sigma(u^0)\overline{\Delta} W_0,
\end{align}
where $(u^0,v^0) = (u_0,v_0)$.  With this choice of $u^1$,  it is easy to verify that
\begin{align}\label{eq3.4}
\mE\bigl[\|\Delta(u(t_1) - u^1)\|^2_{L^2}\bigr] \leq C\tau,
\end{align}
where $C>0$ is a constant. This estimate will be used in the proof of Theorem \ref{theorem32} later.

\item[(b)] Our schemes defined by \eqref{theta_scheme} are quite different from those proposed in \cite{feng2022higher}. 
\end{itemize}
\end{remark}

\subsection{Stability estimates}
This subsection is devoted to the stability analysis for our $\theta$-scheme. Recall that it consists of two schemes corresponding to  $\theta=0$ and $\theta=\frac12$.  To the end, for each $1 \leq n \leq N-1$,  we introduce 
the forward difference operator $\delta_t u^n:= (u^{n+1} - u^n)/\tau$. 

Our main stability results are summarized in the following lemma. 

\begin{lemma}\label{lemma3.1}
Let $\{(u^n, v^n)\}$ denote the numerical solutions generated by the $\theta$-scheme \eqref{eq20230703_30}-\eqref{eq20230703_31}. Under the assumptions ${\bf (A1)-(A3)}$, there exists a constant $C= C(u_0,v_0,D,T)>0$ such that
\begin{enumerate}
\item[(a)] For $\theta =0$,  there holds
\begin{align}\label{ineq3.5}
\max_{1 \leq n \leq N}\mE\Bigl[\|\delta_t u^n\|^2_{L^2} + \|\nabla u^n\|^2_{ L^2} + \|v^n\|^2_{L^2}\Bigr] \leq C.
\end{align}
\item[(b)] For $\theta = \frac12$,  there holds
\begin{align}\label{ineq3.6}
\max_{1 \leq n \leq N}\mE\Bigl[ \|\nabla u^n\|^2_{L^2} + \|v^n\|^2_{L^2}\Bigr] \leq C,
\end{align}
\end{enumerate}
\end{lemma}

\begin{proof}
To show \eqref{ineq3.5}, substituting \eqref{eq20230703_30} into \eqref{eq20230703_31}, we obtain
\begin{align}\label{eq3.7}
&\bigl(\delta_t u^{n+1} - \delta_tu^n, \phi\bigr) + \tau\bigl(\nabla u^{n+1},\nabla \phi\bigr) = \tau \bigl(F(u^{n+1}),\phi \bigr)  \\\nonumber
&\qquad + \bigl(\sigma(u^n)\overline{\Delta} W_n, \phi\bigr)- \frac{1}{\tau}\bigl(\sigma(u^n)\widehat{\Delta} W_n,\phi\bigr) + \frac{1}{\tau}\bigl(\sigma(u^{n-1})\widehat{\Delta}W_{n-1},\phi\bigr).
\end{align}
Taking $\phi = \delta_t u^{n+1}$ in \eqref{eq3.7} yields 
\begin{align}\label{eq3.8}
\bigl(\delta_t u^{n+1} &- \delta_tu^n, \delta_tu^{n+1}\bigr) + \bigl(\nabla u^{n+1},\nabla (u^{n+1} - u^n)\bigr)\\\nonumber
&  = \tau \bigl(F(u^{n+1}),\delta_tu^{n+1} \bigr) + \bigl(\sigma(u^n)\overline{\Delta} W_n, \delta_tu^{n+1}\bigr)\\\nonumber
&\qquad- \frac{1}{\tau}\bigl(\sigma(u^n)\widehat{\Delta} W_n,\delta_tu^{n+1}\bigr)+ \frac{1}{\tau}\bigl(\sigma(u^{n-1})\widehat{\Delta}W_{n-1},\delta_tu^{n+1}\bigr)\\\nonumber
& = I + II + III + IV.
\end{align}
The left-hand side of \eqref{eq3.8} can be written as follows:
\begin{align*}
\bigl(\delta_t u^{n+1} - \delta_tu^n, \delta_tu^{n+1}\bigr) &= \frac12\bigl[\|\delta_tu^{n+1}\|^2_{L^2} - \|\delta_tu^n\|^2_{L^2} + \|\delta_tu^{n+1} - \delta_tu^n\|^2_{L^2}\bigr],\\\nonumber
\bigl(\nabla u^{n+1},\nabla (u^{n+1} - u^n)\bigr) &= \frac{1}{2}\bigl[\|\nabla u^{n+1}\|^2_{L^2} - \|\nabla u^n\|^2_{L^2} + \|\nabla(u^{n+1} - u^n)\|^2_{l^2}\bigr].
\end{align*}

Next, we estimate the right-hand side of \eqref{eq3.8} from above. 
By ${\bf(A2)}$, we have
\begin{align*}
\mE[I] &\leq \frac{C\tau}{2}\mE\bigl[\|u^{n+1}\|^2_{L^2}\bigr] + \frac{\tau}{2}\mE\bigl[\|\delta_tu^{n+1}\|^2_{L^2}\bigr].
\end{align*}	
Using the fact that $\mE\bigl[\bigl(\sigma(u^n)\widehat{\Delta} W_n, \delta_tu^n\bigr)\bigr] =0$ and ${\bf(A2)}$, we get 
\begin{align*}
\mE[II] &= \mE\bigl[\bigl(\sigma(u^n)\overline{\Delta} W_n, \delta_tu^{n+1} - \delta_tu^n\bigr) \bigr]\\\nonumber
&\leq C\tau\mE\bigl[\|u^n\|^2_{L^2}\bigr] + \frac18\|\delta_tu^{n+1} - \delta_tu^n\|^2_{L^2}.
\end{align*}
Similarly, we have
\begin{align*}
\mE[III + IV] &= -\frac{1}{\tau}\mE\bigl[\bigl(\sigma(u^n)\widehat{\Delta}W_n, \delta_tu^{n+1} - \delta_tu^n\bigr)\bigr] \\\nonumber
&\qquad + \frac{1}{\tau}\mE\bigl[\bigl(\sigma(u^{n-1})\widehat{\Delta}W_{n-1}, \delta_tu^{n+1} - \delta_tu^n\bigr)\bigr] \\\nonumber
&\qquad \frac{1}{\tau}\mE\bigl[\bigl(\sigma(u^{n-1})\widehat{\Delta}W_{n-1}, \delta_tu^{n} - \delta_tu^{n-1}\bigr)\bigr].
\end{align*}
By ${\bf (A2)}$ and Remark $\ref{increment}(ii)_2$ for any $n \geq 1$, we obtain
\begin{align*}
\mE[III + IV] \leq C\tau\mE\bigl[\|u^n\|^2_{L^2} + \|u^{n-1}\|^2_{L^2}\bigr] &+ \frac{1}{8}\mE\bigl[\|\delta_tu^{n+1} - \delta_tu^n\|^2_{L^2}\bigr] \\\nonumber
&+ \frac{1}{8}\mE\bigl[\|\delta_tu^{n} - \delta_tu^{n-1}\|^2_{L^2}\bigr].
\end{align*}

Collecting the above  estimates for $I-IV$ and substituting them into \eqref{eq3.8},  we have 
\begin{align}\label{eq3.9}
&\frac12\mE\bigl[\|\delta_tu^{n+1}\|^2_{L^2} - \|\delta_tu^n\|^2_{L^2}\bigr] + \frac14\mE\bigl[\|\delta_tu^{n+1} - \delta_tu^{n}\|^2_{L^2}\bigr] \\\nonumber
&\quad + \frac12\mE\bigl[\|\nabla u^{n+1}\|^2_{L^2} - \|\nabla u^n\|^2_{L^2}\bigr] + \frac12\mE\bigl[\|\nabla(u^{n+1} - u^n)\|^2_{L^2}\bigr] \\\nonumber
&\qquad \leq C\tau \mE\bigl[\|u^{n+1}\|^2_{L^2} + \|u^n\|^2_{L^2} + \|u^{n-1}\|^2_{L^2}\bigr] + C\tau\mE\bigl[\|\delta_tu^{n+1}\|^2_{L^2}\bigr] \\\nonumber
&\quad \qquad+ \frac{1}{8}\mE\bigl[\|\delta_tu^{n} - \delta_tu^{n-1}\|^2_{L^2}\bigr].
\end{align}
Applying the summation operator $\sum_{n = 1}^{\ell}$ ($0\leq \ell\leq N - 1$) to \eqref{eq3.9} leads to 
\begin{align}\label{eq3.10}
&\mE\bigl[\|\delta_tu^{\ell+1}\|^2_{L^2} \bigr] +\mE\bigl[\|\nabla u^{\ell+1}\|^2_{L^2}\bigr] \\\nonumber
&\quad + \sum_{n = 1}^{\ell} \mE\Bigl[\|\delta_tu^{n+1} - \delta_tu^n\|^2_{L^2} + \|\nabla(u^{n+1} - u^n)\|^2_{L^2}\Bigr]\\\nonumber
&\qquad \leq C\mE\bigl[\|\nabla u^0\|^2_{L^2} \bigr] + C\tau\sum_{n = 1}^{\ell}\mE\Bigl[\|\delta_tu^{n+1}\|^2_{L^2} + \|\nabla u^{n+1}\|^2_{L^2}\Bigr].
\end{align}

Then, $\eqref{ineq3.5}_1, \eqref{ineq3.5}_2$ are followed by applying the discrete Gronwall's inequality to \eqref{eq3.10}.  Finally,  $\eqref{ineq3.5}_3$ follows from \eqref{eq20230703_30} with help of $\eqref{ineq3.5}_1$ and $\eqref{ineq3.5}_2$. 

\medskip
To show \eqref{ineq3.6}, noticing that \eqref{eq20230703_30} can be rewritten as
\begin{align*}
\Bigl(\frac{u^{n+1} - u^{n-1}}{\tau},\phi\Bigr) &=\Bigl( \frac{u^{n+1} - u^n}{\tau} + \frac{u^{n} - u^{n-1}}{\tau},\phi\Bigr)\\
&= \bigl(v^{n+1} + v^n,\phi\bigr) -\frac{1}{\tau} \bigl(\sigma(u^n)\widehat{\Delta}W_n, \phi\bigr) \\\nonumber
&\qquad- \frac{1}{\tau}\bigl(\sigma(u^{n-1})\widehat{\Delta}W_{n-1},\phi\bigr).
\end{align*}
Taking $\phi = (u^{n+1} - u^{n-1})/\tau$ in \eqref{eq20230703_31} with $\theta = \frac12$, we obtain
\begin{align}\label{eq3.11}
&\bigl(v^{n+1} - v^n, v^{n+1} + v^n\bigr)	+ \frac12 \bigl(\nabla(u^{n+1} + u^{n-1}), \nabla(u^{n+1}  - u^{n-1})\bigr)\\\nonumber
&= \frac{1}{\tau}\bigl(\sigma(u^n)\widehat{\Delta}W_n + \sigma(u^{n-1})\widehat{\Delta}W_{n-1}, v^{n+1} - v^n\bigr) \\\nonumber
&\qquad+ \Bigl(\sigma(u^n)\overline{\Delta}W_n, \frac{u^{n+1}- u^{n-1}}{\tau}  \Bigr) \\\nonumber
&\qquad+ \Bigl(D_u\sigma(u^n)v^n \widehat{\Delta}W_n , \frac{u^{n+1} - u^{n-1}}{\tau}\Bigr)\\\nonumber
&\qquad + \frac{\tau}{2}\Bigl(F(u^{n+1}) + F(u^{n-1}), \frac{u^{n+1} - u^{n-1}}{\tau}\Bigr)\\\nonumber
&= \mathcal{I}_1 + \mathcal{I}_2 + \mathcal{I}_3 + \mathcal{I}_4.
\end{align}

The left-hand side of \eqref{eq3.11} can be written as
\begin{align*}
\bigl[\|v^{n+1}\|^2_{L^2} - \|v^n\|^2_{L^2}\bigr] + \frac12 \bigl[\|\nabla u^{n+1}\|^2_{L^2} - \|\nabla u^{n-1}\|^2_{L^2}\bigr].
\end{align*}

Next,  we bound the right-hand side of \eqref{eq3.11} from above.  To the end,  choosing $\phi = \sigma(u^n)\widehat{\Delta}W_n + \sigma(u^{n-1})\widehat{\Delta}W_{n-1}$ in \eqref{eq20230703_31} and then substituting it into $\mathcal{I}_1$ we get
\begin{align*}
\mE[\mathcal{I}_1] &= \frac{1}{\tau}\mE\bigl[\bigl(\sigma(u^n)\widehat{\Delta}W_n + \sigma(u^{n-1})\widehat{\Delta}W_{n-1}, v^{n+1} - v^n\bigr)\bigr] \\\nonumber
& = -\frac12\mE\Bigl[\bigl(\nabla(u^{n+1} + u^{n-1}), \nabla\bigl(\sigma(u^n)\widehat{\Delta}W_n + \sigma(u^{n-1})\widehat{\Delta}W_{n-1}\bigr)\bigr)\Bigr]\\\nonumber
&\qquad+ \frac12\mE\Bigl[\bigl(F(u^{n+1}) + F(u^{n-1}), \sigma(u^n)\widehat{\Delta}W_n + \sigma(u^{n-1})\widehat{\Delta}W_{n-1}\bigr)\Bigr]\\\nonumber
&\qquad +\frac{1}{\tau}\mE\Bigl[\bigl(\sigma(u^n)\overline{\Delta}W_{n}, \sigma(u^n)\widehat{\Delta}W_n + \sigma(u^{n-1})\widehat{\Delta}W_{n-1}\bigr)\Bigr] \\\nonumber
&\qquad \qquad+\frac{1}{\tau}\mE\Bigl[\bigl(D_u\sigma(u^n)v^n\widehat{\Delta}W_n, \sigma(u^n)\widehat{\Delta}W_n + \sigma(u^{n-1})\widehat{\Delta}W_{n-1}\bigr) \Bigr]\\\nonumber
&= \mathcal{I}_{1,1} +  \mathcal{I}_{1,2} +  \mathcal{I}_{1,3} +  \mathcal{I}_{1,4}.   
\end{align*}

It follows from Remark $\ref{increment}(\rm ii)_2$ and ${\bf (A1)}$ that
\begin{align*}
\mathcal{I}_{1,1} &\leq C\tau \mE\bigl[\|\nabla u^{n+1}\|^2_{L^2} + \|\nabla u^{n-1}\|^2_{L^2}\bigr]\\\nonumber
&\qquad + \frac{1}{\tau} \mE\bigl[\|\nabla\bigl(\sigma(u^n)\widehat{\Delta}W_n + \sigma(u^{n-1})\widehat{\Delta}W_{n-1}\bigr)\|^2_{L^2}\bigr]\\\nonumber
&\leq C\tau \mE\bigl[\|\nabla u^{n+1}\|^2_{L^2} + \|\nabla u^{n-1}\|^2_{L^2}\bigr] + C\tau^2  \mE\bigl[\|\nabla u^{n+1}\|^2_{L^2} + \|\nabla u^{n-1}\|^2_{L^2}\bigr]\\\nonumber
&\leq C\tau  \mE\bigl[\|\nabla u^{n+1}\|^2_{L^2} + \|\nabla u^{n-1}\|^2_{L^2}\bigr].
\end{align*}
Using ${\bf (A1)}$ and the Poincar\'e inequality, we obtain
\begin{align*}
\mathcal{I}_{1,2} \leq C\tau  \mE\bigl[\|\nabla u^{n+1}\|^2_{L^2} + \|\nabla u^{n-1}\|^2_{L^2}\bigr].
\end{align*}

Similarly, using ${\bf (A1), (A3)},$ and the Poincar\'e inequality, we also get 
\begin{align*}
\mathcal{I}_{1,3} + \mathcal{I}_{1,4} \leq C\tau \mE\bigl[\|v^n\|^2_{L^2} + \|\nabla u^{n}\|^2_{L^2}\bigr].
\end{align*}
To summarize, we obtain 
\begin{align*}
\mE\bigl[\mathcal{I}_1\bigr] \leq C\tau\mE\bigl[\|\nabla u^{n+1}\|^2_{L^2} + \|\nabla u^n\|^2_{L^2} + \|\nabla u^{n-1}\|^2_{L^2}\bigr] + C\tau\mE\bigl[\|v^n\|^2_{L^2}\bigr].
\end{align*}

Now, we continue to bound $\mathcal{I}_2$.  By using the martingale property of the increments $\overline{\Delta}W_n$ and \eqref{eq20230703_31} with $\phi = \sigma(u^n)\overline{\Delta}W_{n}$, we have
\begin{align*}
\mE\bigl[\mathcal{I}_2\bigr] &= \frac{1}{\tau}\mE\bigl[\bigl(\sigma(u^n)\overline{\Delta}W_n, u^{n+1} - u^{n-1}\bigr)\bigr] \\\nonumber
&= \mE\bigl[\bigl(\sigma(u^n)\overline{\Delta}W_n, v^{n+1} + v^n\bigr)\bigr]\\\nonumber
&= \mE\bigl[\bigl(\sigma(u^n)\overline{\Delta}W_n, v^{n+1}\bigr)\bigr]\\\nonumber
&=  \mE\bigl[\bigl(\sigma(u^n)\overline{\Delta}W_n, v^{n+1} - v^n\bigr)\bigr]\\\nonumber
&=  -\frac{\tau}{2} \mE\bigl[\bigl(\nabla\sigma(u^n)\overline{\Delta}W_n, \nabla(u^{n+1} + u^{n-1})\bigr)\bigr] \\\nonumber
&
\qquad+  \frac{\tau}{2}\mE\bigl[\bigl(\sigma(u^n)\overline{\Delta}W_n, F(u^{n+1}) + F(u^{n-1})\bigr)\bigr] + \mE\bigl[\|\sigma(u^n)\overline{\Delta}W_n\|^2_{L^2}\bigr]\\\nonumber
&\qquad+  \mE\bigl[\bigl(\sigma(u^n)\overline{\Delta}W_n, D_u\sigma(u^n)v^n\widehat{\Delta}W_n\bigr)\bigr]\\\nonumber
&=  \mathcal{I}_{2,1}  +  \mathcal{I}_{2,2} +  \mathcal{I}_{2,3} +  \mathcal{I}_{2,4}.   
\end{align*}
Each of $ \mathcal{I}_{2,1},  \mathcal{I}_{2,2}, \mathcal{I}_{2,3},$ and $\mathcal{I}_{2,4}$ can be controlled using the same techniques for bounding  $ \mathcal{I}_{1,1},  \mathcal{I}_{1,2}, \mathcal{I}_{1,3}, $ and $\mathcal{I}_{1,4}$ above (we skip the repetitions to save space). In summary, we have
\begin{align*}
\mE\bigl[\mathcal{I}_2\bigr] \leq C\tau\mE\bigl[\|\nabla u^{n+1}\|^2_{L^2} + \|\nabla u^n\|^2_{L^2} + \|\nabla u^{n-1}\|^2_{L^2}\bigr] + C\tau\mE\bigl[\|v^n\|^2_{L^2}\bigr].
\end{align*}

To bound $\mathcal{I}_3$, using the martingale property of the increments, ${\bf (A1)}$, \eqref{eq20230703_30} with $\varphi = D_u\sigma(u^n)v^n\widehat{\Delta}W_n$, we obtain
\begin{align*}
\mE\bigl[\mathcal{I}_3\bigr] &= \frac{1}{\tau}\mE\bigl[\bigl(D_u\sigma(u^n)v^n\widehat{\Delta}W_n, u^{n+1} - u^n\bigr)\bigr]\\\nonumber
&= \mE\bigl[\bigl(D_u\sigma(u^n)v^n\widehat{\Delta}W_n, v^{n+1}\bigr)\bigr] - \frac{1}{\tau}\mE\bigl[\bigl(D_u\sigma(u^n)v^n\widehat{\Delta}W_n, \sigma(u^n)\widehat{\Delta}W_n\bigr)\bigr]\\\nonumber
&\leq C\tau\mE\bigl[\|v^{n+1}\|^2_{L^2} + \|v^n\|^2_{L^2}\bigr] + C\tau\mE\bigl[\|\nabla u^n\|^2_{L^2}\bigr].
\end{align*}

Finally, we bound $\mathcal{I}_4$ as follows.  Using \eqref{eq20230703_30} with test function  $\varphi = F(u^{n+1}) + F(u^{n-1})$ and applying ${\bf (A1)}$ yield
\begin{align}
\mE\bigl[\mathcal{I}_4\bigr] &= \frac{\tau}{2}\mE\bigl[\bigl(F(u^{n+1}) + F(u^{n-1}), v^{n+1} + v^n\bigr)\bigr] \\\nonumber
&\qquad- \frac12\mE\bigl[\bigl(F(u^{n+1}) + F(u^{n-1}), \sigma(u^n)\widehat{\Delta}W_n\bigr)\bigr] \\\nonumber
&\qquad- \frac12\mE\bigl[\bigl(F(u^{n+1}) + F(u^{n-1}), \sigma(u^{n-1})\widehat{\Delta}W_{n-1}\bigr)\bigr]\\\nonumber
&\leq C\tau\mE\bigl[\|v^{n+1}\|^2_{L^2} + \|v^n\|^2_{L^2}\bigr] \\\nonumber
&\qquad+ C\tau\mE\bigl[\|\nabla u^{n+1}\|^2_{L^2} + \|\nabla u^n\|^2_{l^2} + \|\nabla u^{n-1}\|^2_{L^2}\bigr].
\end{align} 

Collecting all the estimates for $\mathcal{I}_1, \mathcal{I}_2, \mathcal{I}_3, \mathcal{I}_4$ and substituting them into \eqref{eq3.11} followed  by taking summation from $n = 0$ to $n = \ell$ for some $0 \leq \ell <N$, we get
\begin{align}\label{eq3.13}
\mE\bigl[\|v^{\ell +1}\|^2_{L^2} + \|\nabla u^{\ell+1}\|^2_{L^2}\bigr] &\leq C\mE\bigl[\|v^0\|^2_{L^2} + \|\nabla u^0\|^2_{L^2}\bigr] \\\nonumber
&\,\,+ \tau\sum_{n=0}^{\ell}\mE\bigl[\|v^{n+1}\|^2_{L^2} + \|v^n\|^2_{L^2}\bigr]\\\nonumber
&\,\,+\tau\sum_{n=0}^{\ell}\mE\bigl[\|\nabla u^{n+1}\|^2_{L^2} + \|\nabla u^n\|^2_{L^2} + \|\nabla u^{n-1}\|^2_{L^2}\bigr].
\end{align}
The proof is complete by applying the discrete Gronwall's inequality to \eqref{eq3.13}.
\end{proof}


\section{Error estimates}\label{sec-4} 

This section is devoted to establishing the main results of this paper, which are concerned with rates of convergence for two time-discrete schemes proposed in the previous section. This will be done separately because different techniques will be needed for $\theta=0$ and $\theta=\frac12$. 

\subsection{Error estimates for the scheme with $\theta =0$}
Before stating our error estimate results, we need a couple of enhanced  H\"older continuity estimates for the variational solution $(u,v)$. We note that these results are new and of independent interest. 

\begin{lemma}\label{lem20230703_7}
Let $(u,v)$ be the variational solution of \eqref{sac_s_mixed} and let $(u^n,v^n)$ be the solution of the $\theta$-scheme with $\theta = 0$. Then, there holds the following estimate:
\begin{align*}
\E\bigl[\|\tau v\left(t_{n+1}\right)-\left[u\left(t_{n+1}\right)-u\left(t_n\right)\right] &-\sigma({u^{n}})\widehat{\Delta} W_n\|_{L^2}^2\bigr]\\
&\le C\tau^3\E\left[\|e_u^n\|_{L^2}^2+\|e_v^n\|_{L^2}^2\right]+C\tau^4, \notag
\end{align*}
where $e_u^n := u(t_n)-u^n$ and $e_v^n := v(t_n) - v^n$ for $n = 0,1,2,\cdots N$.
\end{lemma}

\begin{proof}
By \eqref{enhanced_holder} (with an obvious adjustment) and the definition of $\widehat{\Delta} W_n$, we have
\begin{align*}
\tau v\left(t_{n+1}\right) &-\left[u\left(t_{n+1}\right)-u\left(t_n\right)\right]
-\sigma({u^{n}})\widehat{\Delta} W_n\label{eq20230703_1}\\\nonumber
&=\int_{t_n}^{t_{n+1}}\int_s^{t_{n+1}} \bigl[\Delta u(\lambda)+F(u(\lambda))d\lambda \bigr]\,ds\\\nonumber 
&\qquad +\left[\int_{t_n}^{t_{n+1}} \int_s^{t_{n+1}} \sigma(u(\lambda)) d W_\lambda d s-(\sigma({u^{n}})\widehat{\Delta} W_n\right]\notag\\
&=\int_{t_n}^{t_{n+1}}\int_s^{t_{n+1}} \bigl[  \Delta u(\lambda)+F(u(\lambda)) \bigr]\, d\lambda ds\\\nonumber
&\qquad +\int_{t_n}^{t_{n+1}} \int_s^{t_{n+1}} \bigl[ \sigma(u(\lambda))\mp\sigma(u(t_n))-\sigma({u^{n}}) \bigr]\,d W_\lambda d s\notag\\
&\qquad +\sigma({u^{n}})(\widetilde{\Delta}W_n-\widehat{\Delta} W_n)\notag.
\end{align*}
Therefore, by using Lemma \ref{increment}(i), we obtain
\begin{align}
&\E\left[\|\tau v\left(t_{n+1}\right)-\left[u\left(t_{n+1}\right)-u\left(t_n\right)\right]-(\sigma({u^{n}})\widehat{\Delta} W_n\|_{L^2}^2\right]\\
&\hskip 0.85in \le C\tau^3\E\left[\|e_u^n\|_{L^2}^2+\|e_v^n\|_{L^2}^2\right] +C\tau^4.\nonumber
\end{align}
The proof is complete. 
\end{proof}

\begin{lemma}\label{lem20230703_1}
Let $(u,v)$ be the variational solution of \eqref{sac_s_mixed} and $(u^n,v^n)$ be the solution of the $\theta$-scheme with $\theta = 0$. Then, there holds the following estimate: 
\begin{align*}
&\E\left[\|\tau \nabla v\left(t_{n+1}\right)-\nabla\left[u\left(t_{n+1}\right)-u\left(t_n\right)\right]-\nabla\sigma({u^{n}})\widehat{\Delta} W_n\|_{L^2}^2\right]\\
&\hskip 0.9in  \le C\tau^3\E\left[\|\nabla e_u^n\|_{L^2}^2+\|\nabla e_v^n\|_{L^2}^2\right]+C\tau^4.\notag
\end{align*}
\end{lemma}
Since the proof is similar to that of Lemma \ref{lem20230703_7}, we omit it to save space and leave it to the interested reader to explore. 

We are now ready to state the first main theorem of this paper. 

\begin{theorem}\label{thm20230704_1}
Let $(u,v)$ be the variational solution of \eqref{sac_s_mixed} and $(u^n,v^n)$ be the solution of the $\theta$-scheme with $\theta = 0$. Then, under the assumptions $({\bf A1})-({\bf A3})$, there holds the following error estimate:
\begin{align}
\max_{1 \leq n \leq N}\Bigl\{\mathbb{E}\left[\|u(t_n) - u^n\|_{L^2}^2\right]&+\mathbb{E}\left[\|\nabla (u(t_n) - u^n)\|_{L^2}^2\right]\\\nonumber
&+\mathbb{E}\left[\|v(t_n) - v^n\|_{L^2}^2\right] \Bigr\}	\le C\tau^2.
\end{align}
\end{theorem}

\begin{proof}
\eqref{sac_mixed_weakform} implies that there holds $\mathbb{P}$-almost surely  that
\begin{align}
& (u(t_{n+1}) - u(t_n), w) = \int_{t_n}^{t_{n+1}} (v(s), w) ds
\qquad\qquad\qquad\quad \forall w \in H_0^1(D),  \label{eq20230703_2} \\
& (v(t_{n+1}) - v(t_n), z) + \int_{t_n}^{t_{n+1}} (\nabla u(s),
\nabla z) ds   \label{eq20230703_3}\\ 
& \qquad = \int_{t_n}^{t_{n+1}}(F(u(s)), z) ds +
\int_{t_n}^{t_{n+1}} (\sigma(u(s)), z) dW(s) \quad \forall z\in H_0^1(D). \notag
\end{align}

Setting $\theta = 0$ in \eqref{eq20230703_30}--\eqref{eq20230703_31} and subtracting the resulted equations from \eqref{eq20230703_2}--\eqref{eq20230703_3} yield
the following error equations:  
\begin{align} 
(e_u^{n+1} - e_u^n, w) &= \int_{t_n}^{t_{n+1}} (v(s) -
v^{n+1}, w) ds  \label{eq20230703_4}\\
&\quad  +(\sigma({u^{n}})\,\widehat{\Delta}W_n,w)\qquad\forall w \in H_0^1(D),\notag \\
(e_v^{n+1} - e_v^{n}, z) &+ \int_{t_n}^{t_{n+1}}(\nabla u(s) - \nabla u^{n+1}, \nabla z) ds \label{eq20230703_5}\\
&= \int_{t_n}^{t_{n+1}}  ( F(u(s)) - F(u^{n+1}), z) ds\notag\\
&\quad+ \int_{t_n}^{t_{n+1}} (\sigma(u(s)) - \sigma(u^n), z) dW(s)\qquad \forall z \in H_0^1(D).\notag
\end{align}
Setting  $w = e_u^{n+1}$ in \eqref{eq20230703_4} and $z = e_v^{n+1}$ in \eqref{eq20230703_5}, and then taking the expectation, we obtain 
\begin{align} 
\E\left[(e_u^{n+1} - e_u^n, e_u^{n+1})\right] &=\E\left[ \int_{t_n}^{t_{n+1}} (v(s) -
v^{n+1}, e_u^{n+1}) ds\right]  \label{eq20230703_12}\\
&\qquad +\E\left[(\sigma({u^{n}})\,\widehat{\Delta}W_n,e_u^{n+1})\right],\notag \\
\E\left[ (e_v^{n+1} - e_v^{n}, e_v^{n+1}) \right] &+\E\left[  \int_{t_n}^{t_{n+1}}(\nabla u(s) - \nabla u^{n+1}, \nabla e_v^{n+1}) ds\right] \label{eq20230703_13}\\
&= \E\left[\int_{t_n}^{t_{n+1}}  ( F(u(s)) - F(u^{n+1}), e_v^{n+1}) ds\right]\notag\\
&\qquad + \E\left[\int_{t_n}^{t_{n+1}} (\sigma(u(s)) - \sigma(u^n), e_v^{n+1}) dW(s)\right]  \notag.
\end{align}

The left-hand side of \eqref{eq20230703_12} can be written as
\begin{align} \label{eq20230703_14}
\mathbb{E}\left[(e_u^{n+1} - e_u^n, e_u^{n+1})\right] & = \frac12\mathbb{E}\left[\|e_u^{n+1}\|_{L^2}^2\right] - \frac12\mathbb{E}\left[\|e_u^n\|_{L^2}^2\right] \\
&\qquad+ \frac12\mathbb{E}\left[\|e_u^{n+1}-e_u^n\|_{L^2}^2\right]\notag.
\end{align}
By Lemma \ref{lem20230703_7}, the right-hand side of \eqref{eq20230703_12} can be bounded from above as follows.
\begin{align}\label{eq20230703_15}
&\E\left[ \int_{t_n}^{t_{n+1}} (v(s) - v^{n+1}, e_u^{n+1}) ds\right] + \E\Bigl[(\sigma({u^{n}})\,\widehat{\Delta}W_n,e_u^{n+1})\Bigr]\\
&\, =\E\Bigl[(u(t_{n+1})-u(t_n) - \tau v(t_{n+1})+\sigma({u^{n}})\,\widehat{\Delta}W_n, e_u^{n+1})\Bigr]+\tau\E\left[(e_v^{n+1},e_u^{n+1})\right]\notag\\
&\,\le C\tau^2\E\left[\|e_u^{n}\|_{L^2}^2\right]+C\tau^3+\tau\mathbb{E}\left[\|e_u^{n+1}\|_{L^2}^2\right]+\tau\mathbb{E}\left[\|e_v^{n+1}\|_{L^2}^2\right]\notag.
\end{align}

The first term on the left-hand side of \eqref{eq20230703_13} can be rewritten as
\begin{align} \label{eq20230703_16}
\mathbb{E}\left[(e_v^{n+1} - e_v^n, e_v^{n+1})\right] 
&=\frac12\mathbb{E}\Bigl[  \|e_v^{n+1}\|_{L^2}^2 -  \|e_v^n\|_{L^2}^2 
+ \|e_v^{n+1}-e_v^n\|_{L^2}^2 \Bigr]. 
\end{align}
The second term on the left-hand side of \eqref{eq20230703_13} can be rewritten as 
\begin{align} \label{eq20230703_17}
&\mathbb{E}\left[\int_{t_n}^{t_{n+1}}(\nabla u(s)
- \nabla u^{n+1}, \nabla e_v^{n+1}) ds\right]\\
&\, =\,\mathbb{E}\left[\int_{t_n}^{t_{n+1}}(\nabla u(s)
- \nabla u(t_{n+1}), \nabla e_v^{n+1}) ds\right] + \mathbb{E}\left[\int_{t_n}^{t_{n+1}}(\nabla e_u^{n+1}, \nabla e_v^{n+1}) ds\right].\notag
\end{align}

Using Lemma \ref{lemma2.3}$(\rm iii)_1$, the first term on the right-hand side of \eqref{eq20230703_17} can be bounded as follows:
\begin{align} \label{eq20230703_18}
\mathbb{E}\left[\int_{t_n}^{t_{n+1}}(\nabla u(s) - \nabla u(t_{n+1}), \nabla e_v^{n+1}) ds\right] 
&=\mathbb{E}\left[\int_{t_n}^{t_{n+1}}(\Delta u(s) - \Delta u(t_{n+1}), e_v^{n+1}) ds\right] \notag \\
&\le  C\tau^3 + \tau\mathbb{E}\left[\|e_v^{n+1}\|_{L^2}^2\right]\notag.
\end{align}
On noting that \eqref{eq20230703_30} implies 
\begin{align} \label{eq20230703_19}
v^{n+1} = \delta_tu^{n+1} + \frac1\tau\sigma({u^{n}})\,\widehat{\Delta}W_n. 
\end{align}
Then, by \eqref{eq20230703_19}, we have
\begin{align} \label{eq20230703_20}
\nabla e_v^{n+1} &=  \nabla v(t_{n+1}) - \nabla v^{n+1}\\
&=\left[\nabla v(t_{n+1}) -\nabla(\delta_tu(t_{n+1}))-\frac1\tau\nabla \sigma({u^{n}})\,\widehat{\Delta}W_n\right]\notag\\
&\qquad+\left[\nabla(\delta_tu(t_{n+1})) - \nabla(\delta_tu^{n+1})\right]\notag.
\end{align}
For the second term on the right-hand side of \eqref{eq20230703_17}, using Lemma \ref{lem20230703_1}, we get
\begin{align} \label{eq20230703_21}
&\mathbb{E}\left[\int_{t_n}^{t_{n+1}}(\nabla e_u^{n+1}, \nabla(v(t_{n+1})) - \nabla(\delta_tu(t_{n+1}))-\frac1\tau\nabla\sigma({u^{n}})\,\widehat{\Delta}W_n ds\right]\\\nonumber
&\qquad +\mathbb{E}\left[\int_{t_n}^{t_{n+1}}(\nabla e_u^{n+1}, \nabla \delta_t e_u^{n+1}) ds\right]\\\nonumber
&\quad  \le \tau\mathbb{E}\left[\|\nabla(v(t_{n+1})) - \nabla(\delta_tu(t_{n+1}))-\frac1\tau\nabla\sigma({u^{n}})\,\widehat{\Delta}W_n\|_{L^2}^2\right] \\ \nonumber
&\qquad +C\tau\mathbb{E}\left[\|\nabla e_u^{n+1}\|_{L^2}^2\right] 
+\frac12\mathbb{E}\left[\|\nabla e_u^{n+1}\|_{L^2}^2 \right]-\frac12\mathbb{E}\left[\|\nabla e_u^{n}\|_{L^2}^2 \right] \\
&\qquad  + \frac12\mathbb{E}\left[\|\nabla e_u^{n+1}-\nabla e_u^n\|_{L^2}^2 \right]\notag\\\nonumber
&\quad \le C\tau^3+C\tau^2\mathbb{E}\left[\|\nabla e_u^{n+1}\|_{L^2}^2\right] +\frac12\mathbb{E}\left[\|\nabla e_u^{n+1}\|_{L^2}^2 \right]-\frac12\mathbb{E}\left[\|\nabla e_u^{n}\|_{L^2}^2 \right] \\\nonumber
&\qquad + \frac12\mathbb{E}\left[\|\nabla e_u^{n+1}-\nabla e_u^n\|_{L^2}^2 \right]\notag.
\end{align}

Next, using the assumption $({\bf A2})_{1}$ and Lemma \ref{lemma2.3}(i), the first term on the right-hand side of \eqref{eq20230703_13} can be bounded as follows:
\begin{align}\label{eq20230703_22} 
&	\mathbb{E}\left[\int_{t_n}^{t_{n+1}}  ( F(u(s)) - F^{n+1}, e_v^{n+1}) ds\right] \\
&\quad = \, \mathbb{E}\left[\int_{t_n}^{t_{n+1}}  ( F(u(s)) - F(u(t_{n+1})), e_v^{n+1}) ds\right] \notag \\
&\qquad \quad + \mathbb{E}\left[\int_{t_n}^{t_{n+1}}  ( F(u(t_{n+1}) - F(u^{n+1}), e_v^{n+1}) ds\right]\notag\\
&\quad \le C\tau^3 + C\tau\mathbb{E}\left[\|e_v^{n+1}\|_{L^2}^2\right] + C\tau\mathbb{E}\left[\|e_u^{n+1}\|_{L^2}^2\right].\notag
\end{align}

Finally, using  It\^{o} isometry, the second term on the right-hand side of \eqref{eq20230703_13} can be bounded as 
\begin{align}\label{eq20230703_23}
&\mathbb{E}\left[\int_{t_n}^{t_{n+1}} (\sigma(u(s)) - \sigma(u^n), e_v^{n+1}) dW(s)\right]\\
&\quad = \mathbb{E}\left[\int_{t_n}^{t_{n+1}} (\sigma(u(s)) - \sigma(u(t_n)), e_v^{n+1}-e_v^n) dW(s)\right]\notag\\
&\qquad \quad +\mathbb{E}\left[\int_{t_n}^{t_{n+1}} (\sigma(u(t_n))-\sigma(u^n), e_v^{n+1}-e_v^n) dW(s)\right]\notag\\
&\qquad \le \frac14\mathbb{E}\left[\|e_v^{n+1}-e_v^n\|_{L^2}^2\right] + C\tau\mathbb{E}\left[\|e_u^{n}\|_{L^2}^2 \right] + C\tau^3\notag.
\end{align}

Combining estimates obtained  \eqref{eq20230703_12}-\eqref{eq20230703_23} yields 
\begin{align}\label{eq20230704_1}
& \frac12\mathbb{E}\left[\|e_u^{n+1}\|_{L^2}^2\right] - \frac12\mathbb{E}\left[\|e_u^n\|_{L^2}^2\right] + \frac12\mathbb{E}\left[\|e_u^{n+1}-e_u^n\|_{L^2}^2\right] \\
&\qquad +\frac12\mathbb{E}\left[\|e_v^{n+1}\|_{L^2}^2\right] - \frac12\mathbb{E}\left[\|e_v^n\|_{L^2}^2\right] + \frac12\mathbb{E}\left[\|e_v^{n+1}-e_v^n\|_{L^2}^2\right]\notag\\
&\qquad +\frac12\mathbb{E}\left[\|\nabla e_u^{n+1}\|_{L^2}^2\right]-\frac12\mathbb{E}\left[\|\nabla e_u^{n}\|_{L^2}^2 ds\right] + \frac12\mathbb{E}\left[\|\nabla e_u^{n+1}-\nabla e_u^n\|_{L^2}^2\right]\notag\\
&\quad \le  C\tau\mathbb{E}\left[ \|e_u^{n+1}\|_{L^2}^2 + \|e_u^{n}\|_{L^2}^2 \right] \notag \\ 
&\qquad +C\tau\mathbb{E}\left[ \|e_v^{n+1}\|_{L^2}^2 \right] + C\tau\mathbb{E}\left[\|\nabla e_u^{n+1}\|_{L^2}^2\right] \notag \\
&\qquad +\frac14\mathbb{E}\left[ \|e_v^{n+1}-e_v^n\|_{L^2}^2\right] + \frac14 \mathbb{E}\left[\|\nabla e_u^{n+1}-\nabla e_u^n\|_{L^2}^2\right] + C\tau^3 \notag.
\end{align}
Taking the summation over $n$ from $0$ to $\ell-1$ and applying the discrete Gronwall's inequality, we obtain
\begin{align}\label{eq20230704_2}
& \mathbb{E}\left[\|e_u^{\ell}\|_{L^2}^2\right]+\mathbb{E}\left[\|\nabla e_u^{\ell}\|_{L^2}^2 \right]+\mathbb{E}\left[\|e_v^{\ell}\|_{L^2}^2\right] 
+ \sum\limits_{n=0}^{\ell-1}\mathbb{E}\left[\|e_v^{n+1}-e_v^n\|_{L^2}^2\right]\\
&\qquad\qquad\quad    + \sum\limits_{n=0}^{\ell-1}\mathbb{E}\left[\|e_u^{n+1}-e_u^n\|_{L^2}^2\right] 
+ \sum\limits_{n=0}^{\ell-1}\mathbb{E}\left[\|\nabla e_u^{n+1}-\nabla e_u^n\|_{L^2}^2 \right] 
\le C\tau^2. \nonumber
\end{align}
The proof is complete. 		
\end{proof}

\begin{remark}
Although the $\mathcal{O}(\tau)$ rate of convergence in the energy norm for the displacement approximation $u^n$ was obtained previously for other schemes \cite{cohen2013trigonometric, feng2022higher},   the above $\mathcal{O}(\tau)$ convergence rate in the $L^\infty(L^2)$-norm for the velocity approximation $v^n$ is the first of its kind. 
\end{remark}

\subsection{Error estimates for the scheme with $\theta = \frac12$}

The goal of this subsection is to establish an optimal $\mathcal{O}(\tau^{\frac32})$ order of error estimates for the $\theta$-scheme with $\theta = \frac12$, which is the second main theorem of this paper. As expected, the proof of this theorem is much more involved and technical than that of the first theorem for the case $\theta-0$. 

\begin{theorem}\label{theorem32} 
Let $(u,v)$ be the variational solution of \eqref{sac_s_mixed} and  $(u^n,v^n)$ be the solution of the $\theta$-scheme with $\theta = \frac12$. Then, under the assumptions $({\bf A1})-({\bf A3})$ and the initial approximations $\{u^0, v^0, u^1\}$ satisfy \eqref{eq3.3} and \eqref{eq3.4}. Then, there holds the following estimate:
\begin{align}\label{optimal_rate}
\max_{1 \leq n \leq N}\E\bigl[\|u(t_n)-u^n\|^2_{L^2}\bigr]\le C\tau^{3}.
\end{align}
\end{theorem}

\begin{proof} 
The proof is quite long, we divide it into {\em eight} steps for the sake of clarity. 

\medskip 
{\em Step 1:}
First, substituting \eqref{eq20230703_30} into \eqref{eq20230703_31}, we obtain
\begin{align}\label{eq3.23}
&\bigl((u^{m+1} - u^m) - (u^m - u^{m-1}), \phi\bigr) + \bigl(\sigma(u^m) \widehat{\Delta} W_m - \sigma(u^{m-1})\widehat{\Delta} W_{m-1}, \phi\bigr)\\\nonumber
&\qquad \qquad +\tau^2 \bigl(\nabla u^{m,\frac12},\nabla \phi\bigr) \\\nonumber
&\quad = \tau^2 \bigl(F^{m,\frac12},\phi\bigr) + \tau \bigl(\sigma(u^m)\overline{\Delta} W_m, \phi\bigr) + \tau\bigl(D_u\sigma(u^m)v^m\widehat{\Delta}W_m, \phi\bigr).
\end{align}
Applying the summation operator $\sum_{m=1}^n$ ($1\leq n <N$) to \eqref{eq3.23} and denoting $\overline{u}^{n+1} := \sum_{m=1}^n u^{m+1}$, we get 
\begin{align}\label{eq3.24}
&	\bigl(u^{n+1} - u^n, \phi\bigr) + \bigl(\sigma(u^n)\widehat{\Delta} W_n, \phi\bigr) + \tau^2\bigl(\nabla\overline{u}^{n,\frac12},\nabla\phi\bigr)\\\nonumber
&\qquad  = \bigl(u^1 - u^0,\phi\bigr) + \bigl(\sigma(u^0)\widehat{\Delta} W_0 + \tau^2\Delta u^{\frac12},\phi\bigr) + \tau^2\sum_{m=1}^n \bigl(F^{m,\frac12}, \phi\bigr)\\\nonumber 
&\qquad \qquad + \tau\sum_{m=1}^n \bigl(\sigma(u^m),\phi\bigr)\overline{\Delta} W_m +\tau\sum_{m=1}^n\bigl(D_u\sigma(u^m)v^m\widehat{\Delta}W_m, \phi\bigr),
\end{align}
where $u^{\frac12} = u^{0+\frac12} = \frac12\bigl(u^1 + u^0\bigr)$.

Next, integrating \eqref{sac_s_mixed} in time and using the fact that $v = u_t$ we obtain for any $0 \leq \lambda \leq \mu  \leq T$ and $\phi \in H^1_0(D)$
\begin{align}\label{eq3.25}
&	\bigl(u(\mu) - u(\lambda),\phi\bigr) + \int_{\lambda}^{\mu} \int_0^s \bigl(\nabla u(\xi), \nabla \phi\bigr)\, d\xi\,ds\\\nonumber
&\,\,= (\mu - \lambda) \bigl(v_0,\phi\bigr) + \int_{\lambda}^{\mu} \int_0^s \bigl(F(u(\xi)), \phi\bigr)\,d\xi\,ds + \int_{\lambda}^{\mu} \int_0^s \bigl(\sigma(u(\xi),\phi\bigr)\, dW(\xi)\, ds. 
\end{align}
Setting $\mu = t_{n+1}$ and $\lambda = t_n$ in \eqref{eq3.25}, and subtracting \eqref{eq3.25} from \eqref{eq3.24} yield the following error equation:
\begin{align*}
&\bigl(e_u^{n+1} - e^n_{u},\phi\bigr) + \int_{t_n}^{t_{n+1}} \int_0^s \bigl(\nabla u(\xi),\nabla\phi\bigr)\, d\xi\, ds - \tau^2 \bigl(\nabla\overline{u}^{n,\frac12}, \nabla \phi\bigr)\\\nonumber
&\quad = \bigl(\tau v_0 - (u^1 - u^0) - \sigma(u^0)\widehat{\Delta} W_0 - \tau^2\Delta u^{\frac12},\phi\bigr) \\\nonumber
&\qquad+ \int_{t_n}^{t_{n+1}} \int_0^s \bigl(F(u(\xi)),\phi\bigr)\, d\xi\, ds - \tau^2 \sum_{m=1}^n \bigl(F^{m,\frac12} , \phi\bigr)\\\nonumber
&\qquad+ \int_{t_n}^{t_{n+1}} \int_0^s \bigl(\sigma(u(\xi)),\phi\bigr)\, dW(\xi)\, ds - \tau\sum_{m=1}^n \bigl(\sigma(u^m),\phi\bigr)\overline{\Delta} W_m \\\nonumber
&\qquad\qquad\qquad-\tau\sum_{m=1}^n\bigl(D_u\sigma(u^m)v^m\widehat{\Delta}W_m, \phi\bigr)+ \bigl(\sigma(u^n)\widehat{\Delta} W_n, \phi\bigr),
\end{align*}
where $e_u^n := u(t_n)-u^n$ and $e_v^n := v(t_n) - v^n$.  Thus, 
\begin{align}\label{error_equation}
&\bigl(e_u^{n+1} - e^n_{u},\phi\bigr) + \tau^2 \bigl(\nabla \overline{e}_u^{n,\frac12},\nabla\phi\bigr)\\\nonumber
&\quad = \bigl(\tau v_0 - (u^1 - u^0) - \sigma(u^0)\widehat{\Delta} W_0 + \frac{\tau^2}{2}\Delta e_u^1,\phi\bigr) \\\nonumber 
&\qquad \quad - \Bigl(\int_{t_n}^{t_{n+1}}\int_0^s \nabla u(\xi)\, d\xi\, ds - \tau^2\sum_{m=1}^n \frac{\nabla\bigl[u(t_{m+1}) + u(t_{m-1})\bigr]}{2},\nabla\phi\Bigr) \\\nonumber
&\qquad\quad +\Biggl\{ \int_{t_n}^{t_{n+1}} \int_0^s \bigl(F(u(\xi)), \phi\bigr)\, d\xi\, ds - \tau^2 \sum_{m=1}^n \bigl(F^{m,\frac12} , \phi\bigr)   \Biggr\}\\\nonumber
&\qquad\quad + \Biggl\{\int_{t_n}^{t_{n+1}} \int_0^s \bigl(\sigma(u(\xi)),\phi\bigr)\, dW(\xi)\, ds - \tau\sum_{m=1}^n \bigl(\sigma(u^m),\phi\bigr)\overline{\Delta} W_m \\\nonumber
&\qquad\quad -\tau\sum_{m=1}^n\bigl(D_u\sigma(u^m)v^m\widehat{\Delta}W_m, \phi\bigr)+ \bigl(\sigma(u^n)\widehat{\Delta} W_n, \phi\bigr)  \Biggr\}.
\end{align}
Taking $\phi = e_u^{n+\frac12} := \frac{1}{2}(e_u^{n+1} + e_u^n)$ in \eqref{error_equation},  we obtain 
\begin{align}\label{error_27}
&\bigl(e_u^{n+1} - e^n_{u},e_u^{n+\frac12}\bigr) + \tau^2 \bigl(\nabla \overline{e}_u^{n,\frac12},\nabla e_u^{n+\frac12}\bigr)\\\nonumber
&\quad = \bigl(\tau v_0 - (u^1 - u^0) - \sigma(u^0)\widehat{\Delta} W_0 ,e_u^{n+\frac12}\bigr) + \frac{\tau^2}{2}\bigl(\Delta e_u^1, e_u^{n+\frac12}\bigr) \\\nonumber 
&\qquad\quad  - \Bigl(\int_{t_n}^{t_{n+1}}\int_0^s \nabla u(\xi)\, d\xi\, ds - \tau^2\sum_{m=1}^n \frac{\nabla\bigl[u(t_{m+1}) + u(t_{m-1})\bigr]}{2},\nabla e_u^{n+\frac12}\Bigr) \\\nonumber
&\qquad\quad +\Biggl\{ \int_{t_n}^{t_{n+1}} \int_0^s \bigl(F(u(\xi)), e_u^{n+\frac12}\bigr)\, d\xi\, ds - \tau^2 \sum_{m=1}^n \bigl(F^{m,\frac12} , e_u^{n+\frac12}\bigr)   \Biggr\}\\\nonumber
&\qquad\quad + \Biggl\{\int_{t_n}^{t_{n+1}} \int_0^s \bigl(\sigma(u(\xi)), e_u^{n+\frac12}\bigr)\, dW(\xi)\, ds - \tau\sum_{m=1}^n \bigl(\sigma(u^m),  e_u^{n+\frac12}\bigr)\overline{\Delta} W_m \\\nonumber
&\qquad\quad -\tau\sum_{m=1}^n\bigl(D_u\sigma(u^m)v^m\widehat{\Delta}W_m, e_u^{n+\frac12}\bigr)+ \bigl(\sigma(u^n)\widehat{\Delta} W_n, e_u^{n+\frac12}\bigr)  \Biggr\}\\\nonumber
&\quad =: {\tt I + II + III + IV}.
\end{align}

\medskip
{\em Step 2:} 
The two terms on the left-hand side of \eqref{error_27} can be rewritten as follows: 
\begin{align}\label{lhs_first}
\bigl(e_u^{n+1} - e^n_{u},e_u^{n+\frac12}\bigr) &= \frac12\bigl[\|e_u^{n+1}\|^2_{L^2} - \|e_u^{n}\|^2_{L^2}\bigr], \\  \label{lhs_second}
\tau^2 \bigl(\nabla \overline{e}_u^{n,\frac12},\nabla e_u^{n+\frac12}\bigr) &= \frac{\tau^2}{4} \bigl(\nabla\overline{e}_u^{n+1} + \nabla\overline{e}_u^{n-1}, \nabla \overline{e}_u^{n+1} - \nabla\overline{e}_u^{n-1}\bigr)\\\nonumber
&= \frac{\tau^2}{4}\bigl[\|\nabla\overline{e}_u^{n+1}\|^2_{L^2}  - \|\nabla\overline{e}_u^{n-1}\|^2_{L^2}\bigr].
\end{align}

We now bound each term on the right-hand side of \eqref{error_equation} from above. To control the noise term ${\tt IV}$, we need to use the following fact:  
\begin{align*}
\int_0^s \sigma(u(\xi))\, dW(\xi) = \biggl(\sum_{m=0}^{n-1} \int_{t_m}^{t_{m+1}} + \int_{t_n}^s\biggr) \sigma(u(\xi))\, dW(\xi), \qquad t_n \leq s \leq t_{n+1}. 
\end{align*}
Therefore, the first two terms of {\tt IV} can be written as
\begin{align}\label{eq3.29}
&\Bigl(\int_{t_n}^{t_{n+1}}\int_0^s \sigma(u(\xi))\, dW(\xi)\, ds - \tau \sum_{m=1}^n \sigma(u^m)\overline{\Delta} W_m, e_u^{n+\frac12}\Bigr) \\\nonumber
&\quad =  \Bigl(\int_{t_n}^{t_{n+1}} \sum_{m=1}^{n-1} \int_{t_m}^{t_{m+1}} \bigl[\sigma(u(\xi)) - \sigma(u(t_m))\bigr]\, dW(\xi)\, ds, e_u^{n+\frac12}\Bigr)  \\\nonumber
&\qquad\quad +\Bigl(\int_{t_n}^{t_{n+1}} \sum_{m=1}^{n-1} \int_{t_m}^{t_{m+1}} \bigl[\sigma(u(t_m)) - \sigma(u^m)\bigr]\, dW(\xi)\, ds, e_u^{n+\frac12}\Bigr)  \\\nonumber
&\qquad\quad +\Bigl(\int_{t_n}^{t_{n+1}} \int_{t_n}^s \bigl[\sigma(u(\xi)) - \sigma(u(t_n))\bigr]\, dW(\xi)\, ds ,e_u^{n+\frac12}\Bigr)\\\nonumber
&\qquad\quad +\Bigl(\sigma(u(t_n))\int_{t_n}^{t_{n+1}} \bigl(W(s) - W(t_n)\bigr)\,ds, e_u^{n+\frac12}\Bigr) \\\nonumber
&\qquad\quad  +\tau \Bigl(\int_{t_0}^{t_1}\sigma(u(\xi))\, dW(\xi), e_u^{n+\frac12}\Bigr) -  \tau\Bigl( \sigma(u^n)\overline{\Delta} W_n, e_u^{n+\frac12} \Bigr).
\end{align}
Substituting \eqref{eq3.29} for the first two terms of {\tt IV},  we obtain
\begin{align}\label{eq3.30}
{\tt IV} &= \Bigl(\int_{t_n}^{t_{n+1}} \sum_{m=1}^{n-1} \int_{t_m}^{t_{m+1}} \bigl[\sigma(u(\xi)) - \sigma(u(t_m))\bigr]\, dW(\xi)\, ds, e_u^{n+\frac12}\Bigr)  \\\nonumber
&\qquad+\Bigl(\int_{t_n}^{t_{n+1}} \sum_{m=1}^{n-1} \int_{t_m}^{t_{m+1}} \bigl[\sigma(u(t_m)) - \sigma(u^m)\bigr]\, dW(\xi)\, ds, e_u^{n+\frac12}\Bigr)  \\\nonumber
&\qquad+\Bigl(\int_{t_n}^{t_{n+1}} \int_{t_n}^s \bigl[\sigma(u(\xi)) - \sigma(u(t_n))\bigr]\, dW(\xi)\, ds ,e_u^{n+\frac12}\Bigr)\\\nonumber
&\qquad+\Bigl(\sigma(u(t_n))\int_{t_n}^{t_{n+1}} \bigl(W(s) - W(t_n)\bigr)\,ds, e_u^{n+\frac12}\Bigr)\\\nonumber 
&\qquad - \tau \Bigl(\sum_{m=1}^n D_u \sigma(u^m) v^m \widehat{\Delta} W_m, e_u^{n+\frac12}\Bigr) + \Bigl(\sigma(u^n) \widehat{\Delta} W_n, e_u^{n+\frac12}\Bigr)  \\\nonumber
&\qquad+\tau \Bigl(\int_{t_0}^{t_1}\sigma(u(\xi))\, dW(\xi), e_u^{n+\frac12}\Bigr) -  \tau\Bigl( \sigma(u^n)\overline{\Delta} W_n, e_u^{n+\frac12} \Bigr)\\\nonumber
&=: {\tt IV_1 + IV_2 + \cdots + IV_8}.
\end{align}

\medskip
{\em Step 3:} 
Adding and subtracting $D_u \sigma(u(t_n)) v(t_n) (\xi -t_n)$  to the term ${\tt IV_1}$ and combining with the term ${\tt IV_5}$ yield
\begin{align*}
&{\tt IV_1 + IV_5}\\\nonumber
&= \tau\Bigl(\sum_{m=1}^{n-1} \int_{t_m}^{t_{m+1}} \bigl[\sigma(u(\xi)) - \sigma(u(t_m)) - D_u\sigma(u(t_m))v(t_m) (\xi - t_m) \bigr]\, dW(\xi), e_u^{n+\frac12}\Bigr)  \\\nonumber
&\qquad+  \tau\Bigl(\sum_{m=1}^{n-1} \int_{t_m}^{t_{m+1}} \bigl(D_u\sigma(u(t_m)) v(t_m) - D_u \sigma(u^m) v^m\bigr)(\xi - t_m)\, dW(\xi)  ,e_u^{n+\frac12}\Bigr) \\\nonumber 
&\qquad+ \tau \Bigl(\sum_{m=1}^{n-1} D_u\sigma(u^m) v^m \bigl(\widetilde{\Delta} W_m - \widehat{\Delta} W_{m}\bigr),  e_u^{n+\frac12}\Bigr)  - \tau \bigl(D_u\sigma(u^n)v^n \widehat{\Delta} W_n, e_u^{n+\frac12}\bigr)\\\nonumber
&=: {\tt IV_{1,5}^1 + IV_{1,5}^2 + IV_{1,5}^3 + IV_{1,5}^4}.
\end{align*}
By the Mean Value Theorem, we obtain
\begin{align*}
\sigma(u(\xi)) - \sigma(u(t_m)) &- D_u\sigma(u(t_m))v(t_m) (\xi - t_m) \\\nonumber
&= \Bigl(D_u\sigma(u(\zeta)) v(\zeta) - D_u\sigma(u(t_m))v(t_m) \Bigr) (\xi - t_m),
\end{align*}
where $\zeta \in (t_m, \xi)$.  It follows from the It\^o isometry and the above identity that 
\begin{align*}
&\mE\bigl[{\tt IV_{1,5}^1}\bigr] \\\nonumber
&\,\,= \tau\mE\Bigl[ \Bigl(\sum_{m=1}^{n-1} \int_{t_m}^{t_{m+1}} \Bigl(D_u\sigma(u(\zeta)) v(\zeta) - D_u\sigma(u(t_m))v(t_m) \Bigr) (\xi - t_m)\, dW(\xi),  e_u^{n+\frac12}\Bigr) \Bigr] \\\nonumber
&\,\,\leq C\tau \mE\Bigl[\sum_{m=1}^{n-1}\int_{t_m}^{t_{m+1}} \|D_u\sigma(u(\zeta)) v(\zeta) - D_u\sigma(u(t_m))v(t_m)\|^2_{L^2} (\xi - t_m)^2\, d\xi\Bigr]\\\nonumber
&\qquad + C\tau \mE\bigl[\|e_u^{n+1}\|^2_{L^2} + \|e_u^{n}\|^2_{L^2}\bigr]\\\nonumber
&\,\,\leq C\tau \mE\Bigl[\sum_{m=1}^{n-1}\int_{t_m}^{t_{m+1}} \|u(\zeta) - u(t_m)\|^2_{L^2} (\xi - t_m)^2\, d\xi\Bigr]  + C\tau \mE\bigl[\|e_u^{n+1}\|^2_{L^2} + \|e_u^{n}\|^2_{L^2}\bigr]\\\nonumber
&\,\, \leq C\tau^5 + C\tau \mE\bigl[\|e_u^{n+1}\|^2_{L^2} + \|e_u^{n}\|^2_{L^2}\bigr].
\end{align*}
Similarly, we can show
\begin{align*}
\mE\bigl[{\tt IV_{1,5}^2}\bigr] \leq C\tau^4 \sum_{m=1}^{n-1} \mE\bigl[\|e_u^m\|^2_{L^2}\bigr] + C\tau \mE\bigl[\|e_u^{n+1}\|^2_{L^2} + \|e_u^{n}\|^2_{L^2}\bigr].
\end{align*}

To bound ${\tt IV_{1,5}^3}$,  using Remark \ref{increment}, Lemma $\ref{lemma3.1}(b)$, the assumption $({\bf A3})$, and the independence property of the  increments $\{\Delta W_n\}$,  we obtain
\begin{align*}
\mE\bigl[{\tt IV_{1,5}^3}\bigr] &\leq \tau\mE\biggl[\Bigl\|\sum_{m=1}^{n-1} D_{u}\sigma(u^m) v^m \bigl(\widetilde{\Delta}W_m - \widehat{\Delta} W_m\bigr)\Bigr\|_{L^2}\|e_u^{n+\frac12}\|_{L^2}\biggr]\\\nonumber
&\leq C\tau\mE\biggl[\Bigl\|\sum_{m=1}^{n-1} D_{u}\sigma(u^m) v^m \bigl(\widetilde{\Delta}W_m - \widehat{\Delta} W_m\bigr)\Bigr\|^2_{L^2}\biggr] + \tau \mE\bigl[\|e_u^{n+\frac12}\|^2_{L^2}\bigr]\\\nonumber
&\leq C\tau \sum_{m=1}^{n-1} \mE\bigl[\|D_u\sigma(u^m) v^m\bigl( \widetilde{\Delta} W_m - \widehat{\Delta} W_m\bigr) \|^2_{L^2}\bigr] + \tau \mE\bigl[\|e_u^{n+\frac12}\|^2_{L^2}\bigr] \\\nonumber
&\leq C\tau^6\sum_{m=1}^{n-1} \mE\bigl[\|v^m\|^2_{L^2}\bigr] + \tau \mE\bigl[\|e_n^{n+\frac12}\|^2_{L^2}\bigr] \\\nonumber
&\leq C \tau^5 + \tau \mE\bigl[\|e_u^{n+\frac12}\|^2_{L^2}\bigr].
\end{align*}
Similarly, we can show 
\begin{align*}
\mE\bigl[{\tt IV_{1,5}^4}\bigr] &\leq C\tau\mE\bigl[\|D_u\sigma(u^n) v^n \widehat{\Delta} W_n\|^2_{L^2}\bigr] + \tau\mE\bigl[\|e_u^{n+\frac12}\|^2_{L^2}\bigr]\\\nonumber
&\leq C\tau^4 + \tau \mE\bigl[\|e_u^{n+\frac12}\|^2_{L^2}\bigr].
\end{align*}
This then completes bounding ${\tt IV_1 + IV_5}$. 

\medskip 
{\em Step 4:}		
Using Cauchy-Schwarz inequality, the assumption $({\bf A1})_2$, It\^o isometry, and Lemma $\ref{lemma2.3}(\rm i)$,  we get 
\begin{align*}
\mE\bigl[{\tt IV_2}\bigr] &= \mE\biggl[\Bigl(\int_{t_n}^{t_{n+1}} \sum_{m=1}^{n-1} \int_{t_m}^{t_{m+1}} \bigl[\sigma(u(t_m)) - \sigma(u^m)\bigr]\, dW(\xi)\, ds, e^{n+\frac12}_u \Bigr) \biggr]\\\nonumber
&\leq C\tau^2\sum_{m=1}^{n-1}\mE\bigl[\|e^m_u\|^2_{L^2}\bigr] + \tau\mE\bigl[\|e^{n+\frac12}_u\|^2_{L^2}\bigr].\\\nonumber
\mE\bigl[{\tt IV_3}\bigr] &= \mE\biggl[\Bigl(\int_{t_n}^{t_{n+1}} \int_{t_n}^s \bigl[\sigma(u(\xi)) - \sigma(u(t_n))\bigr]\, dW(\xi)\, ds , e_u^{n+\frac12}\Bigr)\biggr]\\\nonumber
&\leq C\tau^4 + \tau\mE\bigl[\|e_u^{n+\frac12}\|^2_{L^2}\bigr].
\end{align*}

Now we bound ${\tt IV_4 + IV_6 + IV_8}$ by using the following rewriting:
\begin{align*}
{\tt IV_4 + IV_6 + IV_8}&= \Bigl(\sigma(u(t_n))\int_{t_n}^{t_{n+1}} \bigl(W(s) - W(t_n)\bigr)\,ds, e_u^{n+\frac12}\Bigr) \\\nonumber
&\qquad+ \Bigl(\sigma(u^n) \widehat{\Delta} W_n, e_u^{n+\frac12}\Bigr) - \tau\Bigl(\sigma(u^n) \overline{\Delta} W_n, e_u^{n+\frac12}\Bigr)\\\nonumber
&= \Bigl(\sigma(u(t_n))\int_{t_n}^{t_{n+1}} \bigl(W(s) - W(t_{n+1})\bigr)\,ds, e_u^{n+\frac12}\Bigr) \\\nonumber
&\qquad+\Bigl(\sigma(u(t_n))\int_{t_n}^{t_{n+1}} \bigl(W(t_{n+1}) - W(t_n)\bigr)\,ds, e_u^{n+\frac12}\Bigr) \\\nonumber
&\qquad+ \Bigl(\sigma(u^n) \widehat{\Delta} W_n, e_u^{n+\frac12}\Bigr) - \tau\Bigl(\sigma(u^n) \overline{\Delta} W_n, e_u^{n+\frac12}\Bigr)\\\nonumber
&=-\Bigl(\sigma(u(t_n)) \widetilde{\Delta} W_n, e_u^{n+\frac12}\Bigr) + \tau\Bigl(\sigma(u(t_n)) \overline{\Delta} W_n, e_u^{n+\frac12}\Bigr) \\\nonumber
&\qquad+ \Bigl(\sigma(u^n) \widehat{\Delta} W_n, e_u^{n+\frac12}\Bigr) - \tau\Bigl(\sigma(u^n) \overline{\Delta} W_n, e_u^{n+\frac12}\Bigr)\\\nonumber
&= \tau \Bigl(\bigl(\sigma(u(t_n)) - \sigma(u^n)\bigr)\, \overline{\Delta} W_n, e_u^{n+\frac12} \Bigr) \\\nonumber
&\qquad- \Bigl(\bigl(\sigma(u(t_n)) - \sigma(u^n)\bigr)\, \widetilde{\Delta} W_n, e_u^{n+\frac12}\Bigr)\\\nonumber
&\qquad + \Bigl(\sigma(u^n) \bigl(\widehat{\Delta} W_n - \widetilde{\Delta} W_n\bigr), e_u^{n+\frac12}\Bigr).
\end{align*}
Taking the expectation on both sides and using the assumption $({\bf A1})_2$ and Lemma  \ref{increment}, we obtain
\begin{align*}
\mE\bigl[{\tt IV_4 + IV_6 + IV_8}\bigr] &= \mE\biggl[\tau \Bigl(\bigl(\sigma(u(t_n)) - \sigma(u^n)\bigr)\, \overline{\Delta} W_n, e_u^{n+\frac12} \Bigr) \\\nonumber
&\qquad- \Bigl(\bigl(\sigma(u(t_n)) - \sigma(u^n)\bigr)\, \widetilde{\Delta} W_n, e_u^{n+\frac12}\Bigr)\\\nonumber
&\qquad + \Bigl(\sigma(u^n) \bigl(\widehat{\Delta} W_n - \widetilde{\Delta} W_n\bigr), e_u^{n+\frac12}\Bigr)\biggr]\\\nonumber
&\leq C\tau^2 \mE\bigl[\|e_u^n\|^2_{L^2}\bigr] + C\tau\mE\bigl[\|e_u^{n+\frac12}\|^2_{L^2}\bigr]  \\\nonumber
&\qquad + \frac{C}{\tau}\mE\bigl[\|\bigl(\sigma(u(t_n)) - \sigma(u^n)\bigr)\, \widetilde{\Delta} W_n\|^2_{L^2}\bigr] + C\tau\mE\bigl[\|e_u^{n+\frac12}\|^2_{L^2}\bigr] \\\nonumber
&\qquad + \frac{C}{\tau} \mE\bigl[\|\sigma(u^n) \bigl(\widehat{\Delta} W_n - \widetilde{\Delta} W_n\bigr)\|^2_{L^2}\bigr] + C\tau\mE\bigl[\|e_u^{n+\frac12}\|^2_{L^2}\bigr] \\\nonumber
&\leq C\tau^2 \mE\bigl[\|e_u^n\|^2_{L^2}\bigr] + C\tau\mE\bigl[\|e_u^{n+\frac12}\|^2_{L^2}\bigr] + C\tau^4\mE\bigl[\|u^n\|^2_{L^2}\bigr] \\\nonumber
&\leq C\tau^2 \mE\bigl[\|e_u^n\|^2_{L^2}\bigr] + C\tau\mE\bigl[\|e_u^{n+\frac12}\|^2_{L^2}\bigr] + C\tau^4.
\end{align*}

Finally, using $({\bf A1})_2$ and Lemma \ref{lemma2.3}$(\rm i )$, we get
\begin{align*}
\mE	\bigl[{\tt IV_7}\bigr]  &=  \mE\Bigl[\tau \Bigl(\int_{t_0}^{t_1}\sigma(u(\xi))\, dW(\xi), e_u^{n+\frac12}\Bigr) \Bigr]\\\nonumber
&=  \mE\Bigl[\tau \Bigl(\int_{t_0}^{t_1}[\sigma(u(\xi)) - \sigma(u(0))]\, dW(\xi), e_u^{n+\frac12}\Bigr) \Bigr] + \tau\mE\bigl[\bigl(\sigma(u(0))\, \overline{\Delta} W_0, e_u^{n+\frac12}\bigr)\bigr] \\\nonumber
&\leq C\tau^4 + \tau\mE\bigl[\|e_u^{n+\frac12}\|^2_{L^2}\bigr] + \tau\mE\bigl[\bigl(\sigma(u(0))\, \overline{\Delta} W_0, e_u^{n+\frac12}\bigr)\bigr].
\end{align*}
We note that the last term 
will be combined with the term {\tt I} in \eqref{error_27} and handled together later.

In summary, we have obtained the following upper bound for $\mE\bigl[{\tt IV}\bigr]$:
\begin{align*}
\mE\bigl[{\tt IV}\bigr] &\leq C\tau^2 \mE\bigl[\|e_u^n\|^2_{L^2}\bigr] + C\tau\mE\bigl[\|e_u^{n+\frac12}\|^2_{L^2}\bigr] + C\tau^2\sum_{m=1}^{n-1}\mE\bigl[\|e^m_u\|^2_{L^2}\bigr]\\\nonumber
&\qquad+ C\tau^4 + \tau\mE\bigl[\bigl(\sigma(u(0))\, \overline{\Delta} W_0, e_u^{n+\frac12}\bigr)\bigr].
\end{align*}

\medskip 
{\em Step 5:}
Our attention now shifts to bounding the term {\tt II} in \eqref{error_27}. To the end, we borrow an idea of \cite{feng2022higher} by using Lemma \ref{approx_integral} to do the job. First, we state the following identity:
\begin{align}\label{eq_3.33}
\int_0^s f(\xi)\,d\xi &:= \biggl(\frac12\sum_{m=1}^n \int_{t_{m-1}}^{t_{m+1}} + \frac12 \int_{t_0}^{t_1} + \frac12 \int_{t_n}^{t_{n+1}} - \int_s^{t_{n+1}}\biggr) f(\xi)\, d\xi\\\nonumber
&=\frac12\biggl(\sum_{m=1}^n \int_{t_{m-1}}^{t_{m+1}} +  \int_{t_0}^{t_1} +  \int_{t_n}^{s} - \int_s^{t_{n+1}}\biggr) f(\xi)\, d\xi.
\end{align}
Using this identity, we obtain
\begin{align}\label{eq4.31}
{\tt II} &= \frac12\int_{t_n}^{t_{n+1}}\sum_{m=1}^n\biggl(  \int_{t_{m-1}}^{t_{m+1}} \nabla u(\xi)\, d\xi -  \frac{2\tau\nabla\bigl[u(t_{m+1}) + u(t_{m-1})\bigr]}{2}, \nabla e_u^{n+\frac12}\biggr) \, ds\\\nonumber
& \qquad+ \frac{1}{2} \int_{t_n}^{t_{n+1}}\biggl(\Bigl(\int_{t_0}^{t_{1}} + \int_{t_n}^s - \int_s^{t_{n+1}}\Bigr) \nabla u(\xi)\, d\xi, \nabla e_u^{n+\frac12}\biggr)\, ds\\\nonumber
&=: {\tt II_1 + II_2}.
\end{align}

Now, let $\Phi(\xi) := \mE\bigl[\bigl(\Delta u(\xi), \nabla e_u^{n+\frac12}\bigr)\bigr]$, then $\Phi'(\xi) =  \mE\bigl[\bigl(\Delta v(\xi), \nabla e_u^{n+\frac12}\bigr)\bigr]$. By Lemma \ref{lemma2.3}$(\rm iv)$, we have 
\begin{align*}
|\Phi'(t) - \Phi'(s)| \leq C\bigl(\mE\bigl[\|e_u^{n+\frac12}\|^2_{L^2}\bigr]\bigr)^{\frac12} \, |t-s|^{\frac12}.
\end{align*}
Therefore, performing integration by parts followed by applying Lemma \ref{approx_integral}, we obtain 
\begin{align*}
\mE\bigl[{\tt II_1}\bigr] &= \frac{\tau}{2}\sum_{m=1}^n \mE\biggl[\biggl(\int_{t_{m-1}}^{t_{m+1}} \nabla u(\xi)\, d\xi -  \frac{2\tau\nabla\bigl[u(t_{m+1}) + u(t_{m-1})\bigr]}{2}, \nabla e_u^{n+\frac12}\biggr)\biggr]\\\nonumber
&=-\tau^2\sum_{m=1}^n \mE\biggl[\biggl(\frac{1}{2\tau}\int_{t_{m-1}}^{t_{m+1}} \Delta u(\xi)\, d\xi -  \frac{\Delta\bigl[u(t_{m+1}) + u(t_{m-1})\bigr]}{2},  e_u^{n+\frac12}\biggr)\biggr]\\\nonumber
&\leq C\tau^{\frac52}\bigl(\mE\bigl[\|e_u^{n+\frac12}\|^2_{L^2}\bigr]\bigr)^{\frac12} \\\nonumber
&\leq C\tau^4 + \tau\mE\bigl[\|e_u^{n+\frac12}\|^2_{L^2}\bigr].
\end{align*}

Next, rewriting  $\mE\bigl[{\tt II_2}\bigr]$ as 
\begin{align*}
\mE\bigl[{\tt II_2}\bigr] &= \frac{\tau}{2}\int_0^{t_1} \mE\bigl[\bigl(\nabla u(\xi), \nabla e_u^{n+\frac12}\bigr)\bigr]\, d\xi \\\nonumber
&\qquad+ \frac12 \int_{t_n}^{t_{n+1}}\mE\biggl[\biggl(\int_{t_n}^s \nabla u(\xi)\, d\xi - \int_s^{t_{n+1}} \nabla u(\xi)\, d\xi,\nabla e_u^{n+\frac12}\biggr)\biggr]\,ds\\\nonumber
&=: {\tt II_2^1 + II_2^2}.
\end{align*}
Adding and subtracting the term $ \nabla u(0)$ and performing integration by parts,  we obtain
\begin{align*}
{\tt II_2^1} &=  -\frac{\tau}{2}\int_0^{t_1} \mE\bigl[\bigl(\Delta (u(\xi)- u(0)),  e_u^{n+\frac12}\bigr)\bigr]\, d\xi  - \frac{\tau^2}{2}\mE\bigl[\big(\Delta u(0), e_u^{n+\frac12}\big)\bigr]\\\nonumber
&\leq C\tau^5 + \tau \mE\bigl[\|e_u^{n+\frac12}\|^2_{L^2}\bigr] - \frac{\tau^2}{2}\mE\bigl[\big(\Delta u(0), e_u^{n+\frac12}\big)\bigr].
\end{align*}
We note that the last term 
will be combined with the term ${\tt I}$ in \eqref{error_27}.

To estimate ${\tt II_2^2}$, by subtracting and adding each of the terms $\frac12\int_{t_n}^s \nabla u(t_n)\, d\xi$ and  $\frac12\int_s^{t_{n+1}} \nabla u(t_n)\, d\xi$, we obtain
\begin{align*}
{\tt II_2^2} &= \frac12 \int_{t_n}^{t_{n+1}}\mE\biggl[\biggl(\int_{t_n}^s \nabla\bigl( u(\xi) - u(t_n)\bigr)\, d\xi,\nabla e_u^{n+\frac12}\biggr)\biggr]\,ds \\\nonumber
&\qquad- \frac12 \int_{t_n}^{t_{n+1}}\mE\biggl[\biggl( \int_s^{t_{n+1}} \nabla \bigl(u(\xi)-u(t_n)\bigr)\, d\xi,\nabla e_u^{n+\frac12}\biggr)\biggr]\,ds\\\nonumber
&\qquad + \frac12 \int_{t_n}^{t_{n+1}}\mE\biggl[\biggl(\int_{t_n}^s \nabla u(t_n)\, d\xi - \int_s^{t_{n+1}} \nabla u(t_n)\, d\xi,\nabla e_u^{n+\frac12}\biggr)\biggr]\,ds.
\end{align*}
Performing integration by parts and using Lemma \ref{lemma2.3}$(\rm iii)$, we obtain
\begin{align*}
{\tt II_2^2} &\leq C\int_{t_n}^{t_{n+1}} \mE\biggl[\Bigl\|\int_{t_n}^s \Delta(u(\xi) - u(t_n))\, d\xi\Bigr\|^2_{L^2}\, ds\biggr] \\\nonumber
&\qquad+ C\int_{t_n}^{t_{n+1}} \mE\biggl[\Bigl\|\int_{s}^{t_{n+1}} \Delta(u(\xi) - u(t_n))\, d\xi\Bigr\|^2_{L^2}\, ds\biggr] + \tau\mE\bigl[\|e_u^{n+\frac12}\|^2_{L^2}\bigr]\\\nonumber
&\qquad+ \frac12\mE\biggl[\biggl(\nabla u(t_n) \int_{t_n}^{t_{n+1}} [(s - t_n) - (t_{n+1} - s)]\, ds, \nabla e_u^{n+\frac12}\biggr)\biggr]\\\nonumber
&\leq C\tau^5 + \tau\mE\bigl[\|e_u^{n+\frac12}\|^2_{L^2}\bigr],
\end{align*}
where we have used the fact that $\int_{t_n}^{t_{n+1}} [(s-t_n) - (t_{n+1} - s)]\, ds  =0$.

In summary,  we have proved that 
\begin{align*}
\mE[{\tt II}] \leq C\tau^4 + \tau\mE\bigl[\|e_u^{n+\frac12}\|^2_{L^2}\bigr] - \frac{\tau^2}{2}\mE\big[\bigl(\Delta u_0, e_u^{n+\frac12}\bigr)\big].
\end{align*}

\medskip
{\em Step 6:}
Similar to the approach for bounding ${\tt II}$,  we rewrite ${\tt III}$ as
\begin{align*}
{\tt III} &= \int_{t_n}^{t_{n+1}} \int_0^s \bigl(F(u(\xi)), e_u^{n+\frac12}\bigr)\, d\xi\, ds - \frac{\tau^2}2 \sum_{m=1}^n \Bigl( F(u(t_{m+1})  + F(u(t_{m-1}))), e_u^{n+\frac12}\Bigr) \\\nonumber
&\qquad+\tau^2\biggl(\sum_{m=1}^n \frac{F(u(t_{m+1})) - F(u^{m+1}) + F(u(t_{m-1})) - F(u^{m-1})}{2},e_u^{n+\frac12}\biggr)\\\nonumber
&=: {\tt III_1 + III_2}.
\end{align*}
It follows from the assumption $({\bf A1})_1$ that
\begin{align*}
\mE\bigl[{\tt III_2}\bigr] \leq C\tau^2\sum_{m=1}^n \bigl[\|e_u^{m+1}\|^2_{L^2} + \|e_u^{m-1}\|^2_{L^2}\bigr] + \tau \mE\bigl[\|e_u^{n+\frac12}\|^2_{L^2}\bigr].
\end{align*}

We now bound ${\tt III_1}$ similarly as we did for ${\tt II_1}$ using the identity \eqref{eq_3.33}. 
\begin{align*}
{\tt III_1} &= \frac12\int_{t_n}^{t_{n+1}}\sum_{m=1}^n\biggl(  \int_{t_{m-1}}^{t_{m+1}}  F(u(\xi))\, d\xi -  \frac{2\tau\bigl[F(u(t_{m+1})) + F(u(t_{m-1}))\bigr]}{2},  e_u^{n+\frac12}\biggr) \, ds\\\nonumber
& \qquad+ \frac{1}{2} \int_{t_n}^{t_{n+1}}\biggl(\Bigl(\int_{t_0}^{t_{1}} + \int_{t_n}^s - \int_s^{t_{n+1}}\Bigr)  F(u(\xi))\, d\xi, e_u^{n+\frac12}\biggr)\, ds\\\nonumber
&=: {\tt III_1^1 + III_1^2}.
\end{align*}
Since 
\begin{align*}
& \left|\left( \partial_tF(u(t))-  \partial_sF(u(s)), e_u^{n+\frac12}\right)\right| \\
& \qquad =\left|\left(D_u F(u(t)) v(t)-D_u F(u(s)) v(s), e_u^{n+\frac12}\right)\right| \\
&\qquad  =\left|\left(\left(D_u F(u(t))-D_u F(u(s))\right) v(t)+D_u F(u(s))(v(t)-v(s)), e_u^{n+\frac12}\right)\right| \\
&\qquad  =\left|\left(\left({D_u^2 F}(\tilde{u}(\rho))(u(t)-u(s))\right)(v(t))+D_u F(u(s))(v(t)-v(s)), e_u^{n+\frac12}\right)\right| \\
&\qquad  \leq {C}\|u(t)-u(s)\|_{{H}^1}\|v(t)\|_{{H}^1} \|e_u^{n+\frac12} \|_{{L}^2}+ {C}\|v(t)-v(s)\|_{{L}^2}\|e_u^{n+\frac12}\|_{{L}^2},
\end{align*}
where $\tilde{u}(\rho) := \rho u(t) + (1-\rho) u(s)$ for some $\rho \in [0, 1]$ and the last inequality is obtained by using the assumption $({\bf A3})_1$, then by Lemma \ref{lemma2.3}$(\rm ii)$ we get 
\begin{align*}
\mE\Bigl[\left|\left( \partial_tF(u(t))-  \partial_sF(u(s)), e_u^{n+\frac12}\right)\right| \Bigr] \leq  C\bigl(\mE\bigl[\|e_u^{n+\frac12}\|^2_{L^2}\bigr]\bigr)^{\frac12} |t-s|^{\frac12}.
\end{align*}

It then follows from Lemma \ref{approx_integral} that
\begin{align*}
\mE\bigl[{\tt III_1^1}\bigr] \leq C\tau^4 + \tau\mE\bigl[\|e_u^{n+\frac12}\|^2_{L^2}\bigr].
\end{align*}

Since ${\tt III_1^2}$ can be bounded following the same lines as we did for ${\tt II_2}$, we omit the derivation. 
In summary, we obtain
\begin{align*}
\mE\bigl[{\tt III}\bigr] &\leq C\tau^4 + C\tau\mE\bigl[\|e_u^{n+\frac12}\|^2_{L^2}\bigr] \\\nonumber
&\qquad+ C\tau^2\sum_{m=1}^n\bigl[\|e_u^{m+1}\|^2_{L^2} + \|e_u^{m-1}\|^2_{L^2}\bigr] - \frac{\tau^2}{2}\mE\bigl[\bigl(F(u_0), e_u^{n+\frac12}\bigr)\bigr].
\end{align*}

\medskip
{\em Step 7:}
Collecting all the terms related to the initial data $(u^0, u^1)$ and combining them with the term {\tt I}, and then using \eqref{eq3.3} and \eqref{eq3.4}, we get 
\begin{align*}
\mE[{\tt \tilde{I}}] &:= \mE[{\tt I}] + \tau\mE\bigl[\bigl(\sigma(u_0)\overline{\Delta}W_0,e_u^{n+\frac12}\bigr)\bigr] - \frac{\tau^2}{2} \mE\bigl[\bigl(\Delta u_0, e_u^{n+\frac12}\bigr)\bigr] - \frac{\tau^2}{2} \mE\bigl[\bigl(F(u_0), e_u^{n+\frac12}\bigr)\bigr]\\\nonumber
&= \mE\bigl[\bigl(\tau v_0 - (u^1 - u^0) - \sigma(u^0)\widehat{\Delta} W_0, e_u^{n+\frac12}\bigr) \bigr]  + \frac{\tau^2}{2}\mE\bigl[\bigl(\Delta e_u^1,e_u^{n+\frac12}\bigr)\bigr]\\\nonumber
&\qquad+ \tau\mE\bigl[\bigl(\sigma(u_0)\overline{\Delta}W_0,e_u^{n+\frac12}\bigr)\bigr] - \frac{\tau^2}{2} \mE\bigl[\bigl(\Delta u_0, e_u^{n+\frac12}\bigr)\bigr] - \frac{\tau^2}{2} \mE\bigl[\bigl(F(u_0), e_u^{n+\frac12}\bigr)\bigr]\\\nonumber
&=  \frac{\tau^2}{2}\mE\bigl[\bigl(\Delta e_u^1,e_u^{n+\frac12}\bigr)\bigr],
\end{align*}
which implies that
\begin{align*}
\mE[{\tt \tilde{I}}] &\leq C\tau^4 + \tau\mE\bigl[\|e_u^{n+\frac12}\|^2_{L^2}\bigr].
\end{align*}

\medskip
{\em Step 8:}
Finally, substituting all the estimates for the terms {\tt I - IV} into the right-hand side of \eqref{error_27} and the identities \eqref{lhs_first} and \eqref{lhs_second} into the left-hand side of \eqref{error_27}, we get
\begin{align}\label{last_form}
&\frac12\mE\bigl[\|e_u^{n+1}\|^2_{L^2} - \|e_u^{n}\|^2_{L^2}\bigr] + \frac{\tau^2}{4}\mE\bigl[\|\nabla\overline{e}_u^{n+1}\|^2_{L^2}  - \|\nabla\overline{e}_u^{n-1}\|^2_{L^2}\bigr] \\  \nonumber
&\qquad \leq C\tau^4 + C\tau\mE\bigl[\|e_u^{n + 1}\|^2_{L^2} + \|e_u^{n}\|^2_{L^2}\bigr] + C\tau^2\sum_{m=1}^n\bigl[\|e_u^{m+1}\|^2_{L^2} + \|e_u^{m-1}\|^2_{L^2}\bigr].
\end{align}
Applying the summation operator $\sum_{n=1}^{\ell}$ ($1 \leq \ell <N$) on both sides and using the discrete Grownwall inequality yield the desired estimate \eqref{optimal_rate}.
\end{proof}

\begin{remark}\
(a)	If the more traditional error estimate techniques of Theorem \ref{thm20230704_1} are applied to the $\theta$-scheme with $\theta = \frac12$, we can instead obtain the following error estimate: 
\begin{align}\label{eq3.47}
\max_{1 \leq n \leq N}\Bigl\{\mathbb{E}\left[\|u(t_n) - u^n\|_{L^2}^2\right]&+\mathbb{E}\left[\|\nabla (u(t_n) - u^n)\|_{L^2}^2\right]\\\nonumber
&+\mathbb{E}\left[\|v(t_n) - v^n\|_{L^2}^2\right] \Bigr\}	\le C\tau^2.
\end{align}
Obviously, the above $\mathcal{O}(\tau)$ order  estimate for $u-u^n$ is weaker than the $\mathcal{O}(\tau^{\frac32})$ result of Theorem \ref{theorem32}, however,  \eqref{eq3.47} also  provides $\mathcal{O}(\tau)$ order  estimate for $v-v^n$, which is new and not provided by Theorem \ref{theorem32}.

(b) Compared to the optimal order scheme proposed in \cite{feng2022higher},  our optimal order scheme with $\theta=\frac12$ is cheaper and easier to implement. 
\end{remark}

\section{Numerical Experiments}\label{sec-5}

\subsection{Finite element spatial discretization}
Since the proposed time-discrete schemes are only semi-discretization in time, we still need to discretize them in both spatial variable $x$ and in stochastic variable $\omega$. For the latter, the classical Monte Carlo method will be employed. For the spatial discretization, we shall use the $P_1$- conforming finite element method. To the end, let  $\mathcal{T}_h$ be a regular triangulation of the domain $\mathcal{D}$ and define the finite element space
\begin{align*}
V_h := \bigl\{v_h \in H^1_0(\mathcal{D}): \  v_h \bigl\vert_K \in \mathcal{P}_1(K) \quad \forall \, K \in \mathcal{T}_h \bigr\}\, ,
\end{align*}
where $\mathcal{P}_1(K)$ denotes the collection of all linear polynomials on $K \in \mathcal{T}_h$. 
The initial data $u^1_h$ is defined by
\begin{align*}\label{u1+v1}
u^1_h = \tau v^0_h + u^0_h -\frac{\tau^2}{2}\Delta_h u^0_h - \frac{\tau^2}{2} Q_h F(u^0) - Q_h\sigma(u^0)\widehat{\Delta} W_0 + \tau Q_h\sigma(u^0)\overline{\Delta} W_0,
\end{align*}
where $(u^0_h,v^0_h) := Q_h(u_0,v_0)$, $Q_h$ denotes the $L^2$-projection operator into the finite element space $V_h$, and $\Delta_h $ stands for the discrete Laplacian operator.  

Our fully discrete $\theta$-scheme then reads as 
\begin{definition*}  
Let $\{ t_n\}_{n=0}^N$ be a uniform mesh of $[0, T]$ with mesh size $\tau>0$, and \eqref{u1+v1}.  For every $n \ge 1$, $\theta=0, \frac12$, find a $[V_h]^2$-valued ${\mathcal F}_{t_{n+1}}$-measurable random field $(u_h^{n+1}, v_h^{n+1})$ such that 
\begin{align*}
&\bigl(u_h^{n+1} - u_h^n, \phi_h) = \tau(v_h^{n+1},\phi_h) - \bigl(\sigma(u^n_h)  \widehat{\Delta} W_n, \phi_h    \bigr)
&\forall \phi_h \in  V_h\, , \\ \notag
&(v_h^{n+1}-v_h^n, \psi_h) +\tau \big(\nabla {u}_h^{n, \theta}, \nabla \psi_h \big)  =\tau  \big(F^{n, \theta},  \psi_h \big)  + \Bigl(\sigma(u^n_h)\overline{\Delta} W_n, \psi_h \Bigr) 	& \\
&\hskip 2.2in +2\theta \Bigl(D_{u} \sigma(u_h^{n} ) v_h^{n} \,\widehat{\Delta}W_n, \psi_h \Bigr) &\forall \psi_h \in  V_h.
\end{align*} 
\end{definition*}
Although we shall not analyze the spatial discretization errors in this paper, it can be shown that $u_h^n$ converges to $u^n$ with the optimal rate $\mathcal{O}(h)$ in the energy norm $\mathbb{E} [ \|\cdot\|_{L^2(H^1)}]$. 

\subsection{Test examples}
In this subsection, we provide three 1-d numerical tests to verify our theoretical error estimates. In all our numerical tests given below, we take $\mathcal{D}=(-1,1)$ and the initial data are take to be $u_0(x)=\sin(2\pi x)$ and $v_0(x)=0$  in \eqref{eq20230423_1}. We also define 
\begin{align*}
&L^{\infty}_tL^2_x(u)= \max_{1 \leq n \leq N}\bigl(\mE\bigl[\|u(t_n) - u^n\|^2_{L^2}\bigr]  \bigr)^{\frac12}, \\
&L^{\infty}_tH^1_x(u)=\max_{1 \leq n \leq N}\bigl(\mE\bigl[\|\nabla(u(t_n) - u^n)\|^2_{L^2}\bigr]  \bigr)^{\frac12},\\
&L^{\infty}_tL^2_x(v)=\max_{1 \leq n \leq N}\bigl(\mE\bigl[\|v(t_n) - v^n\|^2_{L^2}\bigr]  \bigr)^{\frac12}.
\end{align*}
In addition,  we shall use a very small spatial mesh size $h = 1/1024$ so that the spatial error becomes negligible and the time discretization error dominates, which then allows us to verify the theoretical error estimates for the proposed time-discrete schemes. 

\medskip
{\bf Test 1.}  {\bf $\theta$-scheme with $\theta=0$.}
\smallskip

In this test, the drift and diffusion are chosen to be the following simple linear functions: $F(u)=-u$ and $\sigma(u)= u$.

Table \ref{tab1} shows the computed errors in different norms of the $\theta$-scheme with $\theta =0$ and the computed convergence orders. We observe that the temporal error converges with order $1$, which matches the theoretical order of Theorem \ref{thm20230704_1}.

\begin{table}[htb]\label{tab1}
\renewcommand{\arraystretch}{1.2}
\centering 
\footnotesize
\begin{tabular}{|l||c|c||c|c||c|c|}
\hline  
&  $L^{\infty}_tL^2_x(u)$ error \quad & order &   
$L^{\infty}_tH^1_x(u)$ error \quad & order &   
$L^{\infty}_tL^2_x(v)$ error\quad & order \\ \hline    
$\tau=1/2$ & $9.969\times10^{-2}$ & --- &   
$6.264\times10^{-1}$ & --- &   
$1.102\times10^{-1}$ & --- \\ \hline    
$\tau=1/4$ & $5.769\times10^{-2}$ & 0.789 &   
$3.625\times10^{-1}$ & 0.789 &   
$5.268\times10^{-2}$ & 1.065\\ \hline    
$\tau=1/8$ & $2.856\times10^{-2}$ & 1.014 &   
$1.794\times10^{-1}$ & 1.014&   
$2.098\times10^{-2}$ & 1.328\\ \hline    
$\tau=1/16$ & $1.469\times10^{-2}$  & 0.959 &   
$9.231\times10^{-2}$ & 0.959&   
$9.911\times10^{-3}$ &  1.082\\ \hline       
\end{tabular}
\caption{Test 1: Temporal errors and convergence rates at $T = 1$.}
\end{table}

{\bf Test 2.}  {\bf $\theta$-scheme with $\theta=0$.}
\smallskip

In this test, we consider nonlinear drift and diffusion terms, they are chosen as $F(u)=\cos(u)$ and $\sigma(u)=\sin(u)$. 

Table \ref{tab2} displays the computed errors in different norms of the $\theta$-scheme with $\theta =0$ and the computed convergence orders.  It is clear that the temporal errors again converge with order $1$.  
\begin{table}[htb]\label{tab2}
\renewcommand{\arraystretch}{1.2}
\centering 
\footnotesize
\begin{tabular}{|l||c|c||c|c||c|c|}
\hline  
&  $L^{\infty}_tL^2_x(u)$ error \quad & order &   
$L^{\infty}_tH^1_x(u)$ error \quad & order &   
$L^{\infty}_tL^2_x(v)$ error\quad & order \\ \hline    
$\tau=1/2$ & $1.018\times10^{-1}$ & --- &   
$5.890\times10^{-1}$ & --- &   
$1.076\times10^{-1}$ & --- \\ \hline    
$\tau=1/4$ & $5.652\times10^{-2}$ & 0.849 &   
$3.394\times10^{-1}$ & 0.795 &   
$4.344\times10^{-2}$ & 1.308\\ \hline    
$\tau=1/8$ & $3.001\times10^{-2}$ & 0.913 &   
$1.895\times10^{-1}$ & 0.841&   
$1.960\times10^{-2}$ & 1.148\\ \hline    
$\tau=1/16$ & $1.546\times10^{-2}$  & 0.957 &   
$1.035\times10^{-1}$ & 0.872&   
$9.459\times10^{-3}$ &  1.051\\ \hline       
\end{tabular}
\caption{Test 2: Temporal errors and convergence rates at $T = 1 $.}
\end{table}

\medskip
{\bf Test 3.}  {\bf $\theta$-scheme with $\theta=\frac12$.}
\smallskip

Again, we consider the same nonlinear drift and diffusion as in {\bf Test2}, namely, $F(u)=\cos(u)$ and $\sigma(u)=  \sin(u)$.

Table \ref{tab3} shows the computed errors of the numerical solution $(u^n,v^n)$ in different norms, and their convergence orders. We observe that the temporal error for $u^n$ in the $L^2$-norm converges with order $O(\tau^{\frac32})$ as proved in Theorem \ref{theorem32}. In the meantime,  we see an $O(\tau)$ order of convergence for $v^n$ in the $L^2$-norm, which matches well with the theoretical error estimate given in \eqref{eq3.47}. Finally, for comparison purposes, we also computed the temporal error for $u^n$ in the $H^1$-norm. The numerical results also suggest  $O(\tau^{\frac32})$ order of convergence for $u^n$ in the $H^1$-norm, which is not claimed by our error estimates. 

\begin{table}[htb]\label{tab3}
\renewcommand{\arraystretch}{1.2}
\centering 
\footnotesize
\begin{tabular}{|l||c|c||c|c||c|c|}
\hline  
&  $L^{\infty}_tL^2_x(u)$ error \quad & order &   
$L^{\infty}_tH^1_x(u)$ error \quad & order &   
$L^{\infty}_tL^2_x(v)$ error\quad & order \\ \hline    
$\tau=1/2$ & $8.299\times10^{-2}$ & --- &   
$5.395\times10^{-1}$ & --- &   
$8.780\times10^{-2}$ & --- \\ \hline    
$\tau=1/4$ & $2.364\times10^{-2}$ & 1.811 &   
$1.592\times10^{-1}$ & 1.761 &   
$4.006\times10^{-2}$ & 1.132\\ \hline    
$\tau=1/8$ & $7.433\times10^{-3}$ & 1.669 &   
$4.960\times10^{-2}$ & 1.683&   
$1.867\times10^{-2}$ & 1.101\\ \hline    
$\tau=1/16$ & $2.727\times10^{-3}$  & 1.447 &   
$1.829\times10^{-2}$ & 1.439&   
$8.674\times10^{-3}$ &  1.106\\ \hline       
\end{tabular}
\caption{Test 3: Temporal errors and convergence rates at $T=1 $.}
\end{table}


\begin{thebibliography}{99}
		\bibitem {adjerid2011discontinuous} 
		\newblock Adjerid, S., Temimi, H.: 
		\newblock A discontinuous Galerkin method for the wave equation. 
		\newblock Computer Methods in Applied Mechanics and Engineering 200(5-8), 837-849 (2011) 
		
		\bibitem {Anton2016full} 
		\newblock Anton, R., Cohen, D., Larsson, S. \& Wang, X: 
		\newblock Full discretization of Semilinear stochastic wave equations
		driven by multiplicative noise. 
		\newblock SIAM J. Numer. Anal., 54, 1093–1119 (2016) 
		
		\bibitem {chow2015stochastic} 
		\newblock Chow, Pao-Liu:
		\newblock Stochastic partial differential equations, 2nd edn.
		\newblock New York: Chapman \& Hall/CRC (2015)
		
		\bibitem {baccouch2012local} 
		\newblock Baccouch, M.: 
		\newblock A local discontinuous Galerkin method for the second-order wave equation. 
		\newblock Computer Methods in Applied Mechanics and Engineering 209, 129--143 (2012) 
		
		\bibitem {baker1976error} 
		\newblock Baker, G. A.: 
		\newblock Error estimates for finite element methods for second order hyperbolic equations. 
		\newblock SIAM Journal on Numerical Analysis 13(4), 564--576 (1976) 
		
		\bibitem {chou2014optimal} 
		\newblock Chou, C.-S., Shu, C.-W., Xing, Y.: 
		\newblock Optimal energy conserving local discontinuous Galerkin methods for the second-order wave equation in heterogeneous media. 
		\newblock Journal of Computational Physics 272, 88--107 (2014) 
		
		\bibitem {chow2002stochastic} 
		\newblock Chow, P.-L.: 
		\newblock Stochastic wave equations with polynomial nonlinearity. 
		\newblock The Annals of Applied Probability 12(1), 361--381 (2002) 
		
		\bibitem {chow2006asymptotics} 
		\newblock Chow, P.-L.: 
		\newblock Asymptotics of solutions to semilinear stochastic wave equations. 
		\newblock The Annals of Applied Probability 16(2), 757--789 (2006) 
		
		\bibitem {chow2009nonlinear} 
		\newblock Chow, P.-L.: 
		\newblock Nonlinear stochastic wave equations: blow-up of second moments in L2-norm. 
		\newblock The Annals of Applied Probability 19(6), 2039--2046 (2009) 
		
		
		\bibitem {chung2006optimal} 
		\newblock Chung, E.T., Engquist, B.: 
		\newblock Optimal discontinuous Galerkin methods for wave propagation. 
		\newblock SIAM Journal on Numerical Analysis 44(5), 2131--2158 (2006) 
		
		\bibitem {chung2009optimal} 
		\newblock Chung, E.T., Engquist, B.: 
		\newblock Optimal discontinuous Galerkin methods for the acoustic wave equation in higher dimensions. 
		\newblock SIAM Journal on Numerical Analysis 47(5), 3820--3848 (2009) 
		
		\bibitem {cohen2018numerical} 
		\newblock Cohen, D.: 
		\newblock Numerical discretisations of stochastic wave equations. 
		\newblock AIP Conference Proceedings 1, 020001 (2018) 
		
		\bibitem {cohen2013trigonometric} 
		\newblock Cohen, D., Larsson, S., Sigg, M.: 
		\newblock A trigonometric method for the linear stochastic wave equation. 
		\newblock SIAM Journal on Numerical Analysis 51(1), 204--222 (2013) 
		
		\bibitem {cohen2015fully} 
		\newblock Cohen, D., Quer-Sardanyons, L.: 
		\newblock A fully discrete approximation of the one-dimensional stochastic wave equation. 
		\newblock IMA Journal of Numerical Analysis 36(1), 400--420 (2015) 
		
		\bibitem {cui2019strong} 
		\newblock Cui, J., Hong, J., Ji, L., Sun, L.: 
		\newblock Energy-preserving exponential integrable numerical method for stochastic cubic wave equation with additive noise. 
		\newblock arXiv preprint arXiv:1909.00575 (2019) 
		
		\bibitem {dalang2009stochastic} 
		\newblock Dalang, R.C.: 
		\newblock The stochastic wave equation. 
		\newblock A minicourse on stochastic partial differential equations, Lecture Notes in Mathematics book series 1962, 39--71 (2009)
		
		\bibitem {dalang1998stochastic} 
		\newblock Dalang, R.C., Frangos, N.E.: 
		\newblock The stochastic wave equation in two spatial dimensions. 
		\newblock Annals of Probability 26(1), 187--212 (1998) 
		
		\bibitem {Dragomir} 
		\newblock S. S. Dragomir and S. Mabizela,
		\newblock Some error estimates in the trapezoidal quadrature rule, 
		\newblock Tamsui Oxf. J. Math. Sci., 16(2): 259–272 (2000)
		
		\bibitem {dupont19732} 
		\newblock Dupont, T.: 
		\newblock $L^2$-estimates for Galerkin methods for second-order hyperbolic equations. 
		\newblock SIAM Journal on Numerical Analysis 10(5), 880--889 (1973) 
		
		\bibitem {falk1999explicit} 
		\newblock Falk, R.S., Richter, G.R.: 
		\newblock Explicit finite element methods for symmetric hyperbolic equations. 
		\newblock SIAM Journal on Numerical Analysis 36(3), 935--952 (1999) 
		
		\bibitem {feng2022higher} 
		\newblock Feng, Xiaobing and Panda, Akash Ashirbad and Prohl, Andreas,
		\newblock Higher order time discretization for the stochastic semilinear wave equation with multiplicative noise. 
		\newblock IMA Journal of Numerical Analysis (2023) 
		
		\bibitem {grote2006discontinuous} 
		\newblock Grote, M.J., Schneebeli, A., Sch{\"o}tzau, D.: 
		\newblock Discontinuous Galerkin finite element method for the wave equation. 
		\newblock SIAM Journal on Numerical Analysis 44(6), 2408--2431 (2006) 
		
		\bibitem {grote2009optimal} 
		\newblock Grote, M.J., Sch{\"o}tzau, D.: 
		\newblock Optimal error estimates for the fully discrete interior penalty DG method for the wave equation. 
		\newblock Journal of Scientific Computing 40(1-3), 257--272 (2009)
		
		\bibitem {gubinelli2018renormalization} 
		\newblock Gubinelli, M., Koch, H., Oh, T.: 
		\newblock Renormalization of the two-dimensional stochastic nonlinear wave equations. 
		\newblock Transactions of the American Mathematical Society 370(10), 7335--7359 (2018) 
		
		\bibitem {hausenblas2010weak} 
		\newblock Hausenblas, E.: 
		\newblock Weak approximation of the stochastic wave equation. 
		\newblock Journal of computational and applied mathematics 235(1), 33--58 (2010) 
		
		\bibitem {hong2021energy} 
		\newblock Hong, J., Hou, B., Sun, L.: 
		\newblock Energy-preserving fully-discrete schemes for nonlinear stochastic wave equations with multiplicative noise. 
		\newblock arXiv preprint arXiv:2105.14720 (2021) 
		
		\bibitem {li2022finite} 
		\newblock Li, Y. and Wu, S. and Xing, Y.: 
		\newblock Finite element approximations of a class of nonlinear stochastic wave equations with multiplicative noise. 
		\newblock Journal of Scientific Computing 91(2) (2022) 
		
		\bibitem {kovacs2010finite} 
		\newblock Kovács, M., Stig, L., Fardin, S.: 
		\newblock Finite element approximation of the linear stochastic wave equation with additive noise. 
		\newblock SIAM Journal on Numerical Analysis 48(2), 408--427 (2010)  
		
		\bibitem {millet2000stochastic} 
		\newblock Millet, A., Morien, P.L.: 
		\newblock On a stochastic wave equation in two space dimensions: regularity of the solution and its density. 
		\newblock Stochastic processes and their applications 86(1), 141--162 (2000) 
		
		\bibitem {millet1999stochastic} 
		\newblock Millet, A., Marta, S.-S.: 
		\newblock  A stochastic wave equation in two space dimensions: smoothness of the law. 
		\newblock The Annals of Probability 27(2), 803--844 (1999) 
		
		\bibitem {monk2005discontinuous} 
		\newblock Monk, P., Richter, G.R.: 
		\newblock A discontinuous Galerkin method for linear symmetric hyperbolic systems in inhomogeneous media. 
		\newblock Journal of Scientific Computing 22(1-3), 443--477 (2005) 
		
		\bibitem {quer2006space} 
		\newblock Quer-Sardanyons, L., Sanz-Sol{\'e}, M.: 
		\newblock Space semi-discretisations for a stochastic wave equation. 
		\newblock Potential Analysis 24(4), 303--332 (2006) 
		
		\bibitem {riviere2003discontinuous} 
		\newblock Riviere, B., Wheeler, M.F.: 
		\newblock Discontinuous finite element methods for acoustic and elastic wave problems. 
		\newblock Contemporary Mathematics 329(271-282), 4--6 (2003) 
		
		\bibitem {safjan1993high} 
		\newblock Safjan, A., Oden, J.: 
		\newblock High-order Taylor-Galerkin and adaptive hp methods for second-order hyperbolic systems: application to elastodynamics. 
		\newblock Computer Methods in Applied Mechanics and Engineering 103(1-2), 187--230 (1993) 
		
		\bibitem {sun2021} 
		\newblock Sun, Z., Xing, Y.: 
		\newblock Optimal error estimates of discontinuous Galerkin methods with generalized fluxes for wave equations on unstructured meshes. 
		\newblock Mathematics of Computation 90, 1741--1772 (2021) 
		
		\bibitem {walsh2006numerical} 
		\newblock Walsh, J.B.: 
		\newblock On numerical solutions of the stochastic wave equation. 
		\newblock Illinois Journal of Mathematics 50(1-4), 991--1018 (2006) 
		
		\bibitem {xing2013energy} 
		\newblock Xing, Y., Chou, C.-S., Shu, C.-W.: 
		\newblock Energy conserving local discontinuous Galerkin methods for wave propagation problems. 
		\newblock Inverse Probl. Imaging 7(3), 967-986 (2013) 
		
		\bibitem {zhong2011numerical} 
		\newblock Zhong, X., Shu, C.-W.: 
		\newblock Numerical resolution of discontinuous Galerkin methods for time-dependent wave equations. 
		\newblock Computer Methods in Applied Mechanics and Engineering 200(41-44), 2814--2827 (2011) 
	\end{thebibliography}
\end{document}